\documentclass[a4paper,reqno,12pt]{amsart}
\usepackage[utf8]{inputenc}

\usepackage{amsmath, amssymb, amsthm, epsfig}
\usepackage{hyperref, latexsym}
\usepackage{url}
\usepackage[mathscr]{euscript}
\usepackage{color}
\usepackage{harpoon}
\usepackage{url}
\usepackage{mathabx}
\usepackage{tikz}

\usepackage{fullpage} 
\usepackage{setspace}
\usepackage{adjustbox}

\usepackage{mathtools}
\mathtoolsset{showonlyrefs}

\def\today{\ifcase\month\or
  January\or February\or March\or April\or May\or June\or
  July\or August\or September\or October\or November\or December\fi
  \space\number\day, \number\year}

\newtheorem{theorem}{Theorem}

\newtheorem{lemma}[theorem]{Lemma}
\newtheorem{proposition}[theorem]{Proposition}
\newtheorem{remark}[theorem]{Remark}
\newtheorem{definition}[theorem]{Definition}

\newcommand{\A}{\mathcal{A}}

\newcommand{\C}{\mathcal{S}}   %RENAMED TO S - BEST CONSTANT FOR THE SPHERE

\newcommand{\E}{\mathcal{E}}
\newcommand{\F}{\mathcal{F}}

\renewcommand{\H}{\mathcal{H}}

\newcommand{\J}{\mathcal{J}}

\renewcommand{\L}{\mathcal{L}}
\newcommand{\M}{\mathcal{M}}
\newcommand{\Real}{\mathbb{R}}
\renewcommand{\S}{\mathcal{S}}
\newcommand{\T}{\mathcal{T}}

\newcommand{\CC}{\mathfrak{C}}

\newcommand{\GG}{\mathfrak{G}}
\newcommand{\HH}{\mathfrak{H}}

\newcommand{\MM}{\mathfrak{M}}

\newcommand{\RR}{\mathfrak{R}}

\newcommand{\TTT}{\mathbb{T}}

\newcommand{\z}{\mathbb{Z}}

\renewcommand{\r}{\mathbb{R}}
\newcommand{\cp}{\mathbb{C}} %<--- because of my last name

\newcommand{\re}{{\rm Re}\,}
\newcommand{\ft}{\widehat}
\newcommand{\s}{\mathbb{S}}

\newcommand{\bo}{\boldsymbol}

\newcommand{\bn}{{\bo n}}

\newcommand{\wt}{\widetilde}

\newcommand{\la}{\lambda}
\newcommand{\ga}{\gamma}
\newcommand{\al}{\alpha}
\newcommand{\be}{\beta}
\newcommand{\ep}{\varepsilon}

\newcommand{\si}{\sigma}

\newcommand{\p}{\phi}
\newcommand{\dist}{{\rm dist}}

\newcommand{\ov}{\overline}

\newcommand{\Braket}[1]{\left\langle #1\right\rangle}
\newcommand{\Hcal}{\mathcal{H}}
\newcommand{\Scal}{\mathcal{C}_\star}

\newcommand{\Gcal}{\mathcal{G}}
\newcommand{\Pcal}{\mathcal{P}}
\newcommand{\Hdot}{\dot{H}}
\newcommand{\norm}[1]{\lVert #1 \rVert}
\newcommand{\abs}[1]{\left\lvert #1 \right\rvert}
\newcommand{\R}{\mathbb{R}}
\newcommand{\SSS}{\mathbb{S}}
\newcommand{\Z}{\mathbb{Z}}
\newcommand{\fbold}{\boldsymbol{f}}
\newcommand{\gbold}{\boldsymbol{g}}

\newcommand{\sqrtDelta}{\sqrt{-\Delta}}
\newcommand{\DeltaS}{\Delta_{\SSS^d}}
\newcommand{\nud}{\nu_d}
\newcommand{\N}{\mathbb{N}}
\newcommand{\Fhat}{\hat{F}}

\newcommand{\dTu}{\underline{dT}}
\newcommand{\dTp}{\underline{dT}}
\newcommand{\dSu}{\underline{dS}}

\newcommand{\Ical}{\mathcal{I}}
\newcommand{\Fcal}{\mathcal{F}}
\newcommand{\eps}{\epsilon}
\newcommand{\psitilde}{{\widetilde{\psi}}}

\newcommand{\obold}{\boldsymbol{0}}
\newcommand{\rad}{{\rm rad}}

\begin{document}

%------------------HEADINGS------------------------

\title[]{Local Maximizers of Adjoint Fourier Restriction Estimates for the Cone, Paraboloid and Sphere}
\author[Gon\c{c}alves and Negro]{Felipe Gon\c{c}alves and Giuseppe Negro}
\date{\today}
\address{Felipe Gon\c{c}alves, Hausdorff Center for Mathematics, Universit\"at Bonn,
Endenicher Allee 60,
53115 Bonn, Germany }
\email{goncalve@math.uni-bonn.de}
\address{Giuseppe Negro, School of Mathematics, The Watson Building, University of Birmingham, Edgbaston, Birmingham, B15 2TT, England}
\email{g.negro@bham.ac.uk}
\subjclass[2010]{}
\keywords{}
\allowdisplaybreaks
\numberwithin{theorem}{section}
\numberwithin{equation}{section}

%------------------ABSTRACT------------------------------

\begin{abstract}
We show that, possibly after a compactification of spacetime, constant functions are local maximizers of the Tomas-Stein adjoint Fourier restriction inequality for the cone and paraboloid in every dimension, and for the sphere in dimension up to 60. For the cone and paraboloid we work from the PDE framework, which enables the use of the Penrose and the Lens transformations, which map the conjectured optimal functions into constants.
\end{abstract}

%---------------------TITLE--------------------------------

\maketitle

%---------------------HAVE--FUN!-----------------------------------
\section{Introduction}

We consider the Fourier adjoint restriction inequality in the Tomas-Stein endpoint. The general framework is the following; letting $\M\subset \r^{d+1}$ denote a smooth surface with at least $k$ non-vanishing principal curvatures and natural measure $\mu$, and $p = 2+\tfrac{4}{k}$, there is a $C>0$ such that 
\begin{equation}\label{restrictionineq}
\lVert \ft {\mu f} \rVert_{L^p(\r^{d+1})}\le C  \lVert f\rVert_{L^2(\M)}
\end{equation}
for every smooth $f:\M\to \cp$, where 
$$
\ft{\mu f}(x) = \int_\M f(y)e^{-ix\cdot y}d\mu(y).
$$
See~\cite{G,S,To}. Only for a few surfaces $\M$, in specific dimensions, the smallest possible constant $C$ in~\eqref{restrictionineq} and the functions that attain it are known; for many other cases, these are merely conjectured. 

In this paper, we establish a local version of some of these conjectures, according to the following general strategy. We let $\E \subset L^2(\M)$ denote the set of the conjectured extremizers of \eqref{restrictionineq} and we prove that, for all $f$ in a neighborhood of $\E$, 
$$
c\,\dist(f,\E)^2\leq C_\star^2\lVert f\rVert_{L^2(\M)}^2-\lVert \ft {\mu f} \rVert_{L^p(\r^{d+1})}^2 ,
$$
where $c>0$ and $C_\star$ is the conjectured sharp constant. In a few cases, we can prove that this last inequality holds for all $f\in L^2(\M)$, thus establishing a sharpened version of~\eqref{restrictionineq}.

It is well known that estimate \eqref{restrictionineq} is intrinsically connected with dispersion estimates for PDEs, known as Strichartz inequalities; for instance, the Fourier extension operators on the double cone and the paraboloid are connected with solutions of the Wave and Schr\"odinger equations respectively. For more on this see the survey \cite{FO}. 

We proceed by stating our main results.

\subsection{Wave equation - Cone adjoint Fourier restriction}
For $p:=2\frac{d+1}{d-1}$, we consider the estimate 
\begin{equation}\label{eq:StrichWave}
	\left\lVert\cos(t\sqrtDelta)f_0(x)+\frac{\sin(t\sqrtDelta)}{\sqrtDelta}f_1(x)\right\rVert_{L^p(\Real^{1+d})}\!\!\!\le C\norm{f_0, f_1}_{\Hdot^{\frac12}\times \Hdot^{-\frac12}(\Real^d)}, 
\end{equation}
where $t\in \Real, x\in \Real^d$ denote the standard coordinates on $\Real^{1+d}$, and 
\begin{equation}\label{eq:Honehalfnorm}
	\begin{array}{cc}
	\displaystyle \norm{f_0}_{\Hdot^{1/2}}^2:=\int_{\Real^d} \overline f_0\, \sqrtDelta f_0 \, dx, &\displaystyle \norm{f_1}_{\Hdot^{-1/2}}^2:=\int_{\Real^d} \overline f_1\, \frac{1}{\sqrtDelta}f_1\,\, dx.
	\end{array}
\end{equation}
In~\cite{Fo}, Foschi proved that, for $d=3$, the smallest possible value of the constant $C$ is 
\begin{equation}\label{eq:FoschiConstantConj}
	\begin{array}{cc}
		\Scal:=\norm{\cos (t\sqrtDelta) f_\star}_{L^p}=\sqrt{\frac{2}{d-1}} \pi^\frac{1}p \lvert\SSS^d\rvert^{\frac1p-\frac12}, &\text{where } f_\star(x)=c(1+\abs x^2)^{-\frac{d-1}{2}}, 
	\end{array}
\end{equation}
and $c>0$ is chosen so that $\norm{f_\star}_{\Hdot^{1/2}}=1$ (see the forthcoming section~\ref{app:penrose} for the computation of the explicit value of $\Scal$). He conjectured that the same should be true for all dimensions $d\ge 2$, and he also conjectured what the functions $(f_0, f_1)$ that attain this sharp constant should be. 

Before stating our result, we observe that, in the forthcoming Proposition~\ref{prop:real_wave}, we will prove that there is no loss of generality in assuming that $f_0, f_1$ are real-valued. This is relevant to our purposes, because letting 
\begin{equation}\label{eq:complex_notation}
    \begin{array}{cc}
    f:=f_0-\frac{i}{\sqrtDelta} f_1, & \text{ where }f_0, f_1\text{ are real-valued,}
    \end{array}
\end{equation}
it holds that 
\begin{equation}\label{eq:real_consequences} 
    \begin{array}{ccc}
        \displaystyle \cos(t\sqrtDelta)f_0+\frac{\sin(t\sqrtDelta)}{\sqrtDelta}f_1 = \Re(e^{it\sqrtDelta} f), &\text{and}& \norm{f_0, f_1}_{\Hdot^{\frac12}\times \Hdot^{-\frac12}}=\norm{f}_{\Hdot^{\frac12}}.
    \end{array}
\end{equation}
We can thus define the~{deficit functional} of~\eqref{eq:StrichWave} to be 
\begin{equation}\label{eq:wave_deficit}
    \psi(f):=\Scal^2\norm{f}_{\Hdot^\frac12}^2-\norm{\Re(e^{it\sqrtDelta} f)}_{L^p}^2,
\end{equation}
observing that the conjecture of Foschi is equivalent to
\begin{equation}\label{eq:deficit_formulation}
    \begin{array}{cccc}
        \psi\ge 0, & \text{and} & \psi(f)=0 & \text{if and only if }  f\in \MM, 
    \end{array}
\end{equation}
where $\MM\subset \Hdot^{1/2}$ is the manifold of Foschi's conjectured maximizers, namely
\begin{equation}\label{eq:MManifWave}
	\MM:=\left\{ r\Gamma(e^{i\theta}f_\star)\ :\ r> 0,\ \theta\in\mathbb R, \ \Gamma\in \Gcal\right\}, 
\end{equation}
where $\Gcal$ is a Lie group of unitary transformations of $\Hdot^{1/2}$, which we will describe in Section~\ref{sect:wave}. The following is our first main result. 
\begin{theorem}\label{thm:main_wave}
	 Let $d$ be an odd integer. There exist $\delta\in(0,1)$ and $C>0$, depending only on~$d$, and such that
	 \begin{equation}\label{eq:main_wave_result}
	 	\psi(f)\ge C\dist(f, \MM)^2,
	 \end{equation}
	 provided that 
	\begin{equation}\label{eq:wave_distance_cond}
	    \dist(f, \MM)<\delta \norm{f}_{\Hdot^{1/2}(\Real^d)},
	\end{equation}
	where $\dist(f, \MM):=\inf\{ \norm{g-f}_{\Hdot^\frac12}\, :\, g\in \MM\}$. If $d=3$, condition~\eqref{eq:wave_distance_cond} can be removed.
\end{theorem}
An immediate corollary is that the conjecture of Foschi is true, for all odd $d\ge 3$, in an open neighborhood of~$\MM$. The method of proof is based on the study of the first and second derivatives of $\psi$ on $\MM$, using in an essential way the conformal compactification of $\mathbb R^{1+d}$ provided by the Penrose transform. The case $d=3$ is special, because it is already known that $f_\star$ is a global minimizer for $\psi$. Because of this, we can extend the local analysis of the derivatives to a global result, and Theorem~\ref{thm:main_wave} holds without the condition~\eqref{eq:wave_distance_cond}. This is done via a concentration-compactness method dating back to Bianchi and Egnell~\cite{BiEg}, as already observed in~\cite{Negro18}. We remark that, unfortunately, the method of Bianchi and Egnell can only be applied to inequalities for which the extremizers are previously known. 

\begin{figure}
\centering
\begin{tikzpicture}[scale=1.5]
\fill[lightgray] (0:0) -- (10:3.1) -- (350:3.1) -- cycle;
\draw (0:0) -- (4:3);
\draw (0:0) -- (-4:3);
\path (4:3) -- (-4:3) node[midway] (center) {};
\draw (center) ellipse (2.5pt and 8*0.75pt);
\draw (0:3) node[anchor=west]{$\MM$};
\fill (0,0) circle (0.05) node[anchor=south] {$0$};
\fill (4:1.3) circle (0.05) node[anchor=south, yshift=0.7] {$f_\star$};
\end{tikzpicture}
\caption{The Foschi conjecture is true in the shaded region.}
\label{fig:LocalFoschi}
\end{figure}
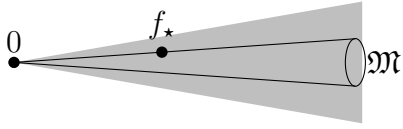

If $d\ge 2$ is even, then $f_\star$ is not a critical point of $\psi$; see the aforementioned~\cite{Negro18}. Thus the conjecture of Foschi fails in this case. However, Theorem~\ref{thm:main_wave} does hold for arbitrary $d\ge 2$ if $\psi$ is replaced by the half-wave functional 
\begin{equation}\label{eq:half_wave_deficit}
    \begin{array}{cc}
\displaystyle     \psi_h(f):=\mathcal{C}_h^2 \norm{f}_{\Hdot^\frac12}^2-\norm{e^{it\sqrtDelta} f}_{L^p}^2, & \text{ where } \mathcal{C}_h:=\norm{e^{it\sqrtDelta} f_\star}_{L^p}.
\end{array}
\end{equation}
For this functional, the condition~\eqref{eq:wave_distance_cond} can be removed for $d=2$ and $d=3$, the cases for which Foschi already proved that $f_\star$ is a global minimizer to $\psi_h$. The propagator $e^{it\sqrtDelta}$ corresponds to the half-wave equation $\partial_t u= i\sqrtDelta u$, or, equivalently, to Fourier adjoint restriction to the one-sheeted cone. It is interesting to note that the distinction between even and odd dimensions, needed for the study of maximizers to the wave Strichartz estimate~\eqref{eq:StrichWave}, does not seem to be needed for the case of the half-wave. To the best of our knowledge, the application of the Penrose transform to the half-wave propagator has not previously appeared in the literature.
%\begin{remark}\label{rem:SharpenedStrichartz} 
%When $d=3$, Theorem~\ref{thm:main_wave} can be sharpened; the inequality~\eqref{eq:main_wave_result} holds for all $f\in \Hdot^{1/2}(\mathbb R^d)$, that is, even if the condition~\eqref{eq:wave_distance_cond} is not satisfied, up to possibly replacing the constant $C$ with a smaller one. This is proven in~\cite{Negro18}, combining a method due to Bianchi and Egnell~\cite{BiEg} with concentration-compactness tools due to Ramos~\cite{Ramos}.
%\end{remark}
% 

\subsection{Schr\"odinger equation - Paraboloid adjoint Fourier restriction}
We consider here the following estimate (see, for example~\cite[Theorem 2.3]{Tao})
\begin{equation}\label{eq:SchrodStrichartz}
\|u\|_{L^{2+4/d}(\mathbb R^{d+1})}\le C\|f\|_{L^2(\mathbb R^d)},
\end{equation}
which holds true for some constant $C>0$, where $u:\r^{r+1}\mapsto\cp$ is the solution of the Schr\"odinger equation
\begin{align*}
& \partial_t u=i\Delta u \\
& u|_{t=0}=f
\end{align*}
with initial data $f$. The smallest possible value of $C$ is given by the supremum 
\begin{equation}\label{eq:SchrodRatio}
   \sup_{f\ne 0}\frac{\|u\|_{L^{2+4/d}(\mathbb R^{d+1})}}{\|f\|_{L^2(\mathbb R^d)}}, 
\end{equation}
and it is a long standing conjecture that $f$ attains this supremum if and only if it is a Gaussian, that is, $f(x)=\exp(a|x|^2 +b \cdot x + c)$, where $c,a\in \cp$, $b\in \cp^d$ and $\re a <0$. If that is the case, a simple calculation shows that such supremum equals
$$
\A_d=4^{-\frac{d}{8+4d}}\left(1+\frac{2}{d}\right)^{-\frac{d^2}{8+4d}}.
$$
We define the {deficit functional} of \eqref{eq:SchrodStrichartz} to be
$$
\p(f)=\A_d^{2}\|f\|_{L^2(\r^d)}^{2}- \|u\|^{2}_{L^{2+4/d}(\r^{d+1})},
$$
and the manifold of Gaussians to be
$$
\GG=\{e^{a|x|^2 +b\cdot x + c}: a,c\in \cp, b\in \cp^d \text{ and } \Re \, a <0\}\subset L^2(\mathbb R^d).
$$
% We define $\T$ as the group of symmetries generated by:
% \begin{itemize}
% \item Space translations $f(x)\mapsto f(x-x_0)$, $x_0\in \r^d$;
% \item Frequency translations $f(x)\mapsto e^{ib\cdot x}f(x)$, $b\in \r^d$;
% \item Finite time propagation (or time translation) $f(x)\mapsto e^{it_1\Delta}(f)(x)$;
% \item Rotations $f(x)\mapsto f(Rx)$, $R\in SO(d)$;
% \item Scaling $f(x)\mapsto \la^{-d/2}f(\la x)$,  $\la>0$.
% \end{itemize}
% One can easily see that $\|Tf\|_{L^2}=\|f\|_{L^2}$  and $\phi(Tf)=\phi(f)$ for every $T\in\T$. 
% Note $\GG=\{e^c\,T(e^{-\pi |\cdot|^2}):T\in\T,c\in \cp\}$. 
We can now state our result.
\begin{theorem}\label{thm:main_schro}
There are $\delta,C>0$, depending only on $d$, and such that 
\begin{equation}\label{eq:SchrodingerConclusion}
\p(f)\geq C\dist(f,\GG)^2
\end{equation}
provided that 
\begin{equation}\label{eq:SchrodDistCond}
\dist(f,\GG)<\delta \|f\|_{L^2}.
\end{equation}
If $d=1$ or $d=2$ condition~\eqref{eq:SchrodDistCond} can be removed.
\end{theorem}
    The method of proof is analogous to the one of the previous section; in this case, the compactifying transformation is a variant of the {Lens transform} (see, for example,~\cite{Tao2}), and we describe it in Section~\ref{sec:lens}. For $d=1, 2$ the removal of the condition~\eqref{eq:SchrodDistCond} is done via the aforementioned method of Bianchi and Egnell; the concentration-compactness tools in this case are due to B\'egout and Vargas~\cite{BeVa}.

\subsection{Sphere adjoint Fourier restriction}
The Tomas-Stein adjoint Fourier restriction inequality asserts that if $p=2\frac{d+1}{d-1}$ then
$$
\sup_{0\ne f\in L^2} \frac{\|\ft {f\sigma}\|_{L^p(\r^d)}}{\|f\|_{L^2(\s^{d-1})}}
 <\infty,
$$
where $\si$ is the surface measure of the sphere $\s^{d-1}$.
%and the supremum is taken over all non-zero $f \in L^2(\s^{d-1})$.
%, in other words, the extension operator $f\in L^2(\s^{d-1})\mapsto \ft{\si f} \in L^p(\r^d)$ is bounded. 
It is conjectured that the only real-valued functions attaining the supremum above are constants, or more generally, functions of the form $f(x)=a e^{ix\cdot v}$, where $a\in \cp\setminus \{0\}$ and $v\in \r^d$. In particular, a simple computation shows that constant functions are maximizers if and only if the supremum above equals
\begin{equation}\label{def:Sconst}
\C_d :=(2\pi)^{d/2}|\s^{d-1}|^{-\frac{1}{d+1}}\left(\int_0^\infty |J_{d/2-1}|^{p} \frac{dr}{r^{\frac{d-3}{d-1}}}\right)^{1/p};
\end{equation}
see the forthcoming~\eqref{hamonicextension}. This conjecture is still widely open and solved only in dimension~3 in a remarkable paper by Foschi \cite{Fo2}. Establishing that constants are locally optimal is usually a more treatable problem, however only in dimension ~2 this was proven \cite{CFOT}. All these results rely heavily on the fact that $p$ is an even integer, a nuance that only happens in dimensions ~2 and ~3. 

We define the deficit function
$$
\zeta(f) =  \C_d^2\|f\|_{L^2(\s^{d-1})}^2 - \|\ft {f\sigma}\|_{L^p(\r^d)}^2
$$
and let 
$$
\CC:=\{a e^{ix\cdot v} : a\in \cp\setminus\{0\} \text{ and } v\in \r^d\}.
$$
\begin{theorem}\label{thm:sphere}
Let $d\leq 60$. There is $\delta,C>0$, depending only on $d$, and such that 
$$
\zeta(f)\geq C \dist(f,\CC)^2
$$
provided that 
\begin{equation}\label{eq:localcondsphereineq}
\dist(f,\CC)<\delta \|f\|_{L^2(\s^{d-1})}.
\end{equation}
%In other words, constants are locally optimal for the adjoint Fourier restriction on $\s^{d-1}$ provided that $d\leq 60$. 
This last condition can be removed for $d=3$.
\end{theorem}
The removal of the condition~\eqref{eq:localcondsphereineq} uses the concentration-compactness tools due to Christ and Shao~\cite{ChSh12}.

Our proof method involves a numerical part that becomes harder as the dimension grows. With more computational power and delicate estimates the threshold on the dimension could possibly be extended to something like $d\leq 200$ and we provide a systematic way of doing so. 
% \rr{maybe include the conjecture about $\ft{|A_\nu|^{p-2}}$ being positive in a sphere of radii 2? Probably in the sphere section after the proof.}
\section{Method of proof}\label{sec:method}

The outline of the proof is common to all cases, so we describe it here in general terms. Let $\mathcal{H}$ be a Hilbert space whose scalar product and norm we denote by $\langle\cdot|\cdot\rangle$ and $\lVert \cdot \rVert$ respectively, and consider the bounded linear operator 
\begin{equation}\label{eq:S_operator}
	\begin{array}{cc}
	\displaystyle	S\colon \mathcal{H}\to L^p(\mathbb R^N), & \text{where }p\in (2, \infty)\text{ and }\ N\in \mathbb N,
	\end{array}
\end{equation}
Let moreover
\begin{equation}\label{eq:G_group}
	\mathcal G:=\{ \Gamma_{\boldsymbol \theta}:\H\to\H \ :\ \boldsymbol\theta \in \mathbb R^k\}
\end{equation}
denote a group of linear transformations depending smoothly on $\boldsymbol \theta$, where $\boldsymbol{\theta}=\boldsymbol 0$. We assume that the following invariance properties are satisfied
\begin{equation}\label{eq:G_invariances}
	\begin{array}{ccc}
		\norm{\Gamma_{\boldsymbol \theta} f}=\norm{f}, &\text{and }\norm{S\Gamma_{\boldsymbol\theta} f}_{L^p}=\norm{Sf}_{L^p} ,& \forall f\in \mathcal H,\ \forall \boldsymbol{\theta}\in\mathbb R^k.
	\end{array}
\end{equation}
Finally, we also assume that $\Gamma_{\boldsymbol \theta}$ {vanishes at infinity}, in the sense that 
\begin{equation}\label{eq:vanish_infinity}
	\lim_{\abs{\boldsymbol \theta} \to \infty} 
	\langle \Gamma_{\boldsymbol \theta}f | g\rangle= 0, \quad \forall f, g\in \mathcal H.
\end{equation}
This concludes the list of the needed assumptions.  

We can now let $f_\star\in\mathcal H\setminus\{0\}$ be fixed and define the {deficit functional} 
\begin{equation}\label{eq:GeneralDeficit}
	\begin{array}{cc}
		\psi(f):=C_\star^2 \norm{f}^2 - \norm{Sf}_{L^p}^2, &\displaystyle \text{ where }C_\star:=\frac{\norm{S f_\star}_{L^p}}{\norm{f_\star}}.
	\end{array}
\end{equation}
By definition, $\psi(f_\star)=0$; hence, by~\eqref{eq:G_invariances}, $\psi$ vanishes on the set
\begin{equation}\label{eq:M_manifold}
	M:=\{ z\Gamma_{\boldsymbol \theta} f_\star\ : z\in\cp\setminus\{0\},\ \boldsymbol\theta\in \mathbb R^k\}\subset \mathcal H.
\end{equation}
By construction,  $M$ is a smooth manifold parametrized by the map 
$$
(z, \boldsymbol \theta)\in \cp\setminus\{0\}\times \mathbb R^{k}\mapsto z\Gamma_{\boldsymbol{\theta}} f_\star;
$$
so, for each $f\in M$, the corresponding tangent space is  
\begin{equation}\label{eq:Tangent_Spaces}	
	T_f M = \operatorname{span}_{\r} \left\{ f, if,\left.\frac{\partial}{\partial \theta_j}\right|_{\theta_j=0}\!\Gamma_{(0,\ldots, \theta_j, \ldots, 0)}f\ :\ j=1, \ldots, k\right\},
\end{equation}
and we denote its orthogonal complement by
\begin{equation}\label{eq:ortho_complement} 
	(T_f M)^\bot :=\left\{ g\in \mathcal H\ :\ \Braket{g, h}=0,\ \forall h\in T_f  M \right\}. 
\end{equation}
We remark that $\psi$ is twice real-differentiable, and we denote\begin{equation}\label{eq:derivative_def}
	\begin{array}{ccc}
		\psi'(f)g:=\left.\frac{\partial}{\partial \eps}\right|_{\eps=0} \psi(f+\eps g), & \psi''(f)(g, g):=\left.\frac{\partial^2}{\partial \eps^2} \right|_{\eps=0} \psi(f+\eps g), & \forall f,g\in \mathcal H.
	\end{array}
\end{equation}
\begin{theorem}\label{thm:general}
	Assume that $\psi'(f_\star)=0$ and that there is a $\rho >0$ such that 
	\begin{equation}\label{eq:positive_tangent}
		\begin{array}{cc}
			\psi''(f_\star)(f_\bot, f_\bot)\ge \rho \norm{f_\bot}^2,& \forall f_\bot \in (T_{f_\star} M)^\bot.
		\end{array}
	\end{equation}
	Then there is a $\delta \in(0, 1) $, depending only on $\rho$, and such that 
	\begin{equation}\label{eq:general_conclusion}
		\psi(f)\ge \frac{\rho}{3}\, \dist(f, M)^2, 
	\end{equation}
	for all $f\in \mathcal H$ such that 
	\begin{equation}\label{eq:dist_condition_general}
	\dist(f, M)<\delta \norm{f}.
	\end{equation}
\end{theorem}
	\begin{remark}\label{rem:RealLinear}
		If $S$ is a real-linear operator, as will be the case in Section~\ref{sect:wave}, the theorem still holds true, with the same proof, provided that, for all $\theta\in\mathbb R$, it holds that $\psi'(e^{i\theta}f_\star)=0$ and that~\eqref{eq:positive_tangent} is replaced by
		\begin{equation}\label{eq:real_positive_tangent}
		\begin{array}{cc}
			\psi''(e^{i\theta} f_\star)(f_\bot, f_\bot)\ge \rho \norm{f_\bot}^2,& \forall f_\bot \in (T_{e^{i\theta}f_\star} M)^\bot.
		\end{array}
	\end{equation}
	\end{remark}
\begin{proof}
	By the vanishing property~\eqref{eq:vanish_infinity}, every $f\in \mathcal H$ has a metric projection on $M\cup \{0\}$; that is, there is at least one $P(f)\in M$, or $P(f)=0$, such that 
	\begin{equation}\label{eq:f_bot_attains}
		\norm{f-P(f)}=\dist(f, M).
	\end{equation} 
	If $\dist(f, M)<\norm{f}$, then $P(f)\ne 0$, thus $P(f)\in M$ and there is the orthogonality $$f-P(f)\in (T_{P(f)} M)^\bot.$$ The proof is standard and can be found, for example, in~\cite{Negro18}.
	
	We let $f_\bot:=f-P(f)$, so that $\norm{f_\bot}=\dist(f, M)$. By the invariances~\eqref{eq:G_invariances}, there is no loss of generality in assuming that $P(f)=z f_\star$ for a $z\in \mathbb C\setminus\{0\}$. Thus, since $\psi(f_\star)=0$ and $\psi'(f_\star)=0$, we can Taylor expand and use~\eqref{eq:positive_tangent} to obtain
	\begin{equation}\label{eq:psi_f}
		\begin{split}
			\psi(f) = \psi(zf_\star + f_\bot) =\abs{z}^2 \psi\left( f_\star + \frac{f_\bot}{z}\right)
			&= \abs{z}^2\left(\frac12 \psi''(f_\star)\left(\frac{f_\bot}{z}, \frac{f_\bot}{z}\right) + o\left( \frac{\lVert f_\bot\rVert^2}{\abs{z}^2}\right)\right) \\ 
			&\ge \abs z^2\left(\frac12 \rho \frac{\norm{f_\bot}^2}{\abs z^2} +o\left( \frac{\lVert f_\bot\rVert^2}{\abs{z}^2}\right)\right) \\ 
			&\ge \frac{\rho}{3}\norm{f_\bot}^2,
		\end{split}
	\end{equation}
	where the last inequality holds provided that $\frac{\lVert f_\bot\rVert^2}{\abs{z}^2}$ is sufficiently small. 
	
	We claim that this smallness is provided by the condition~\eqref{eq:dist_condition_general}. Indeed, $f_\bot\in (T_{f_\star}M)^\bot$, so in particular $\Braket{f_\bot | f_\star}=0$ and so $\norm{f_\perp}^2= \norm{f}^2-\abs{z}^2\norm{f_\star}^2$. Therefore
	\begin{equation}\label{eq:norm_f_ortho}
		\frac{\norm{f_\bot}^2}{\abs{z}^2} = \norm{f_\star}^2\frac{ (\norm{f_\bot}/\norm{f})^2}{1-(\norm{f_\bot}/\norm{f})^2}.
	\end{equation}
	Now, since $x \in [0,1) \mapsto \frac{x^2}{1-x^2}$ is increasing, the left-hand side of~\eqref{eq:norm_f_ortho} is small if and only if $\frac{\norm{f_\bot}}{\norm f}$ is, and this is the case, provided that~\eqref{eq:dist_condition_general} is satisfied with a sufficiently small $\delta\in (0,1)$. This concludes the proof.
	\end{proof}

The condition~\eqref{eq:dist_condition_general} makes the previous result a local one. In some cases, we will be able to upgrade it to a global one, at the cost of some more assumptions.
%\begin{definition}\label{def:ProfDec}
%    We say that the operator $S$ admits a {profile decomposition} if for each sequence $f_n\in \mathcal H$ with $\|f_n\|=1$ for all $n$ and such that $\|Sf_n\|_{L^p}\to C_\star$ there is a {profile} $f\in \mathcal H$, a {sequence of transformations} $\Gamma_{\boldsymbol{\theta}_n}\in\mathcal G$ and a {remainder term} $r_n\in \mathcal H$ such that, up to passing to a subsequence,
%    \begin{equation}\label{eq:profile}
%        f_n=\Gamma_{\boldsymbol{\theta}_n} f + r_n;
%    \end{equation}
%    and moreover, as $n\to \infty$, there hold the orthogonality conditions 
%    \begin{align} 
%        \label{eq:Pythagorean}
%        \norm{f_n}^2&=\norm{f}^2+ \norm{r_n}^2 +o(1),\\ 
%        \label{eq:LpPythagorean} 
%        \norm{Sf_n}_{L^p}^p &= \norm{ Sf}_{L^p}^p + \norm{Sr_n}_{L^p}^p +o(1),
%    \end{align}
%    and the following non-degeneracy condition;
%    \begin{equation}\label{eq:ProfDecNonDeg}
%        \text{if }f=0, \text{ then }\norm{Sr_n}_{L^p}\to 0.
%    \end{equation}
%\end{definition} 
\begin{definition}\label{def:ProfDec}
    We say that maximizing sequences are pre-compact up to symmetries if, for each sequence $f_n\in \mathcal H$ such that $\norm{f_n}=1$ for all $n$, and such that $\norm{Sf_n}_{L^p}\to C_\star$, there are $f, r_n\in \mathcal H$ such that, up to passing to a subsequence,
    \begin{equation}\label{eq:profile}
    	\begin{array}{cc}
        f_n=\Gamma_{\boldsymbol{\theta}_n} f + r_n, & \text{ and }\norm{r_n}\to 0.
        \end{array}
    \end{equation}
\end{definition} 

\begin{definition}\label{def:UniqMaximiz}
    We say that $f_\star$ is a {unique maximizer up to symmetries} if 
    \begin{equation}\label{eq:UniqMaximizOne}
         \frac{\norm{Sf_\star}_{L^p}}{\norm{f_\star}}=\sup\left\{ \frac{\lVert Sf\rVert_{L^p}}{\norm{f}}\ :\ f\ne 0\right\},
    \end{equation}
    and, moreover, for any $g_\star\in \mathcal H$ such that $$\frac{\norm{Sg_\star}_{L^p}}{\norm{g_\star}}=\frac{\norm{Sf_\star}_{L^p}}{\norm{f_\star}},$$ 
    it holds that
    \begin{equation}\label{eq:UniqMaximizTwo}
    	\begin{array}{cc}
        g_\star=z \Gamma f_\star,& \text{for some } \Gamma\in \mathcal G\text{ and } z\in \mathbb C\setminus\{0\}.
        \end{array}
    \end{equation}
    
\end{definition}
    \begin{theorem}[Sharpened inequalities]\label{thm:general_sharpened}
        Assume that the operator $S$ admits pre-compact maximizing sequences and a unique maximizer $f_\star$, up to symmetries. If there is $\rho>0$ such that 
         \begin{equation}\label{eq:positive_tangent_again}
		\begin{array}{cc}
			\psi''(f_\star)(f_\bot, f_\bot)\ge \rho \norm{f_\bot}^2,& \forall f_\bot \in (T_{f_\star} M)^\bot,
		\end{array}
		\end{equation}
		then there is a $C>0$ such that 
		\begin{equation}\label{eq:SharpenedStrichartz}
		    \begin{array}{cc}
		    \psi(f)\ge C\dist(f, M)^2, & \text{for all }f\in\mathcal H.
		    \end{array}
		\end{equation}
		
	\end{theorem}
\begin{proof}
%    Without loss of generality assume that $\norm{f_\star}=\norm{Sf_\star}_{L^p}=1$; in particular, 
%    \begin{equation}\label{eq:norm_smaller_one}
%        \begin{array}{cc}
%        \norm{Sf}_{L^p}\le \norm{f}, & \forall f\in \mathcal H.
%        \end{array}
%    \end{equation}
    Assume, aiming for a contradiction, that there is a sequence $f_n\in\mathcal H$ such that 
    \begin{equation}\label{eq:distance_to_zero}
        \frac{\psi(f_n)}{\dist(f_n, M)^2}\to 0.
    \end{equation}
    By homogeneity, we can assume that $\norm{f_n}=1$. Since $\dist(f_n, M)\le \norm{f_n}$, the denominator in~\eqref{eq:distance_to_zero} is bounded; thus $\psi(f_n)\to 0$, which implies that $f_n$ is a maximizing sequence in the sense of Definition~\ref{def:ProfDec}, so 
    \begin{equation}\label{eq:SpecialProfDecomp}
    	\begin{array}{cc}
        f_n=\Gamma_{\boldsymbol{\theta}_n} f + r_n, & \text{ where }\norm{r_n}\to 0.
        \end{array}
    \end{equation}
%    see Definition~\ref{def:ProfDec}; so, by the orthogonality property~\eqref{eq:LpPythagorean},
%    \begin{equation}\label{eq:Superadditivity}
%        \begin{split}
%            1=\lim_{n\to \infty} \norm{Sf_n}_{L^p}^p&=\norm{Sf}_{L^p}^p + \lim_{n\to \infty} \norm{Sr_n}_{L^p}^p \\ 
%            &\le \norm{f}^p+\lim_{n\to \infty} \norm{r_n}^p.
%        \end{split}
%    \end{equation}
%    Now, since $p>2$, we have the strict inequality $a^p+b^p< (a^2+b^2)^\frac{p}{2}$, if both $a\ne0$ and $b\ne 0$. Thus, if both $f\ne 0$ and $\operatorname*{lim}_{n\to \infty}\norm{r_n}\ne 0$, then~\eqref{eq:Superadditivity} yields the contradiction  
%    \begin{equation}\label{eq:SuperadditivityAftermath}
%        1<(\norm{f}^2+\lim_{n\to \infty}\norm{r_n}^2)^\frac p2=\norm{f_n}^p=1,
%    \end{equation}
%    where we used the orthogonality~\eqref{eq:Pythagorean}. Now, it cannot be $f=0$, because that would imply $\norm{Sr_n}_{L^p}\to 0$ (see~\eqref{eq:ProfDecNonDeg}), contradicting~\eqref{eq:deficit_to_zero}. So we conclude that $\norm{r_n}\to 0$ and $\norm{f}=1$; 
This immediately implies that $f$ is a maximizer for $S$; so by Definition~\ref{def:UniqMaximiz}, $f=z\Gamma_{\boldsymbol{\theta}}f_\star$ for some $z\in \mathbb C$ and $\boldsymbol{\theta}\in\mathbb R^k$. Thus,~\eqref{eq:SpecialProfDecomp} is equivalent to 
    \begin{equation}\label{eq:StabilityOrbital}
        \dist(f_n, M)\to 0.
    \end{equation}
    
    Now, since $\norm{f_n}=1$, for all sufficiently large $n$ it holds that $\dist(f_n, M)<\delta\norm{f_n}$, where $\delta>0$ is the parameter that appears in Theorem~\ref{thm:general}. But then $$\psi(f_n)\ge \frac{\rho}{3}\dist(f_n, M)^2,$$ for a $\rho>0$, contradicting~\eqref{eq:distance_to_zero}. The proof is complete. 
\end{proof}

%To facilitate the application of Theorem~\ref{thm:general}, we record in the following lemma a more explicit expression of its assumptions. The proof is postponed to the Appendix~\ref{app:abstvar}.
%\begin{lemma}\label{lem:Eul_Lagr}
%	Assume that $\norm{f_\star}=\norm{S f_\star}_{L^p}=1$. Then $\psi'(f_\star)=0$ if and only if  
%	\begin{equation}\label{eq:AltEL}
%		\begin{array}{cc}
%		\displaystyle \Re \int_{\mathbb{R}^N} \abs{Sf_\star}^{p-2}\overline{Sf_\star}Sf = \Re\Braket{f_\star | f},& \forall f\in \Hcal.
%		\end{array}
%	\end{equation}
%	If this is verified and, moreover, $f\in \mathcal{H}$ satisfies $\Braket{f|f_\star}=0$ (which is the case if $f\in (T_{f_\star}M)^\perp$), then 
%	\begin{equation}\label{eq:StandardSecOrd}
%	    \begin{array}{ccc}
%		\psi''(f_\star)(f, f)\ge \rho \norm{f}^2 & \text{holds if and only if} &
%		\frac{Q(f, f)}{\norm{f}^2}\le 2-\rho
%		\end{array}
%	\end{equation}
%	where $Q$ is the quadratic form
%	\begin{equation}\label{eq:AltSecOrd}
%		Q(f, f):=p\int_{\mathbb{R}^N} \abs{Sf_\star}^{p-2}\abs{Sf}^2+(p-2)\Re\int_{\mathbb{R}^N}\abs{Sf_\star}^{p-4}(\overline{Sf_\star})^2(Sf)^2.
%	\end{equation}
%	\end{lemma}

\section{Cone adjoint restriction - Wave equation}\label{sect:wave}

We now want to apply the methods of the previous section to the Hilbert space \\ $\mathcal H=\dot{H}^{1/2}\times \dot{H}^{-1/2}(\mathbb R^d)$ and to the operator 
$$
\begin{array}{cc}
	S\colon \mathcal H\to L^p(\mathbb R^{1+d}), & \text{ where } p=2\frac{d+1}{d-1}, 
\end{array}
$$
 given, for $t\in\mathbb R$ and $x\in\mathbb R^d$, by 
\begin{equation}\label{eq:S_oper_wave}
	\begin{array}{ccc}
	S\boldsymbol f(t, x):=\cos(t\sqrtDelta)f_0 (x)+ \frac{\sin(t\sqrtDelta)}{\sqrtDelta} f_1(x), &\text{ where } \boldsymbol f=(f_0, f_1).
	\end{array}
\end{equation}
Before we start the study of the deficit functional, we remark that, if $\fbold$ is real-valued, then $S\fbold$ also is real-valued. Thus, there will be no loss of generality in limiting ourselves to the real-valued case, as the following proposition shows. 
\begin{proposition}\label{prop:real_wave}
	 Suppose that $\fbold, \gbold$ are real-valued and $\fbold+i\gbold\neq \obold$.
	Then either $\fbold=\obold$, or $\gbold=\obold$, or
	\begin{equation}\label{eq:all_maximizers}
	\frac{\norm{S(\fbold+i\gbold)}_{L^p}}{\norm{\fbold+i\gbold}_\Hcal} \leq	\max\left\{\frac{\norm{S\fbold}_{L^p}}{\norm{\fbold}_{\Hcal}},\frac{\norm{S\gbold}_{L^p}}{\norm{\gbold}_{\Hcal}}\right\}.
	\end{equation}
\end{proposition}
\begin{proof}
	We assume  that $\fbold\ne\obold$ and $\gbold\ne \obold$, and moreover 
	\begin{equation}\label{eq:fbold_strict}
		\frac{\norm{ S\fbold}_{L^p}^2}{\norm{\fbold}_{\Hcal}^2} \geq  \frac{\norm{ S\gbold}_{L^p}^2}{\norm{\gbold}_{\Hcal}^2}.
	\end{equation}
	Using that $\norm{\fbold+i\gbold}_\Hcal^2=\norm{\fbold}_\Hcal^2+\norm{\gbold}_\Hcal^2$ and 
	\begin{equation}\label{eq:Lp_triangle}
		\norm{S(\fbold+i\gbold)}_{L^p}^2=\norm{\abs{S\fbold}^2+\abs{S\gbold}^2}_{L^{p/2}}\le \norm{S\fbold}_{L^p}^2+\norm{S\gbold}_{L^p}^2, 
	\end{equation}
	we  infer from \eqref{eq:fbold_strict}
	\begin{equation}\label{eq:final_realization}
		\frac{\norm{S(\fbold+i\gbold)}^2_{L^p}}{\norm{\fbold+i\gbold}_\Hcal^2}\le \frac{\norm{S\fbold}_{L^p}^2+\norm{S\gbold}_{L^p}^2}{\norm{\fbold}^2_\Hcal+ \norm{\gbold}_\Hcal^2} \leq \frac{ (1+\frac{\norm{\gbold}_\Hcal^2}{\norm{\fbold}_\Hcal^2})\norm{S\fbold}_{L^p}^2}{ (1+\frac{\norm{\gbold}_\Hcal^2}{\norm{\fbold}_\Hcal^2})\norm{\fbold}_\Hcal^2}=\frac{\norm{S\fbold}_{L^p}^2}{\norm{\fbold}_\Hcal^2}.
	\end{equation}
	\end{proof}
As stated in the Introduction, for real-valued $\fbold=(f_0, f_1)$, we let 
$f:=f_0-\frac{i}{\sqrtDelta}f_1$, so 
\begin{equation}\label{eq:real_complex_dictionary}
	\begin{array}{cc}
		S\fbold = \Re(e^{it\sqrtDelta} f), &\text{and } \norm{\fbold}_{\Hdot^{1/2}\times \Hdot^{-1/2}}=\norm{f}_{\Hdot^{1/2}};
	\end{array}
\end{equation}
thus, the deficit functional to study reads
\begin{equation}\label{eq:WaveDeficitComplex}
	\begin{array}{c}
	\displaystyle \psi(f)=\Scal^2\norm{f}_{\Hdot^{1/2}}^2 - \lVert \Re(e^{it\sqrtDelta} f)\rVert_{L^p(\Real^{1+d})}^2.
%	& p=2\frac{d+1}{d-1},
	\end{array}
\end{equation}
We recall from~\eqref{eq:FoschiConstantConj} that 
\begin{equation}\label{eq:RecallFstar}
	\begin{array}{cc}
		f_\star(x)=c(1+\abs{x}^2)^{-\frac{d-1}{2}}, & \Scal=\lVert \Re(e^{it\sqrtDelta}f_\star)\rVert_{L^p(\mathbb{R}^{1+d})},
	\end{array}
\end{equation}
where $c>0$ is such that $\norm{f_\star}_{\dot{H}^{1/2}}=1$. 

We now describe the symmetry group. As in Foschi~\cite{Fo}, we let $\Gcal$ denote the Lie group generated by the following transformations of $\Hdot^{1/2}(\Real^d)$;
\begin{equation}\label{eq:LieGroup}
	\begin{array}{cc}
		f\mapsto \lambda^\frac{d-1}{2}e^{it_0\sqrtDelta}f(\lambda Rx+h), & f\mapsto (e^{i(\cdot)\sqrtDelta} f)\circ L^\alpha|_{t=0}, 
	\end{array}
\end{equation}
where $L^\alpha(\tau, \xi)=(\gamma \tau - \gamma \alpha \xi_1, \gamma \xi_1 - \gamma \alpha \tau, \xi_2,\ldots, \xi_d)$, for $\gamma:=(1-\alpha^2)^{-1/2}$, denotes the Lorentzian boost. The parameters in~\eqref{eq:LieGroup} are
\begin{equation}\label{eq:parameters}
	\begin{array}{ccccc}
		t_0\in\Real, &h\in\Real^d, &R\in SO(d),&\lambda >0, &\abs{\alpha}<1.
	\end{array}
\end{equation}
The manifold of conjectured maximizers is, thus,
\begin{equation}\label{eq:ComplexM}
	\MM=\{ r\Gamma(e^{i\theta}f_\star)\ :\ r\ge 0, \theta \in \Real, \Gamma\in\Gcal\}.
\end{equation} 
Differentiating $r\Gamma(e^{i\theta}f_\star)$ with respect to $r, \theta$ and the parameters in~\eqref{eq:parameters}, we compute that the tangent space at $e^{i\theta}f_\star$ is 
\begin{equation}\label{eq:tangent_space}
	T_{e^{i\theta}f_\star} \MM=\mathrm{span}_{\r}\left( f_\star, if_\star, i\sqrtDelta e^{i\theta}f_\star, \partial_{x_j} e^{i\theta}f_\star, \left(\nud+x\cdot\nabla\right)e^{i\theta}f_\star, ix_j\sqrtDelta e^{i\theta}f_\star\right).
\end{equation}
%We remark that, actually, this tangent space is independent of $\theta\in\mathbb R$. 

The group $\Gcal$ satisfies the invariances
\begin{equation}\label{eq:general_invariance}
	\begin{array}{ccc}
		\norm{\Re(e^{i(\cdot)\sqrtDelta}\Gamma f)}_{L^p}=\norm{\Re(e^{i(\cdot)\sqrtDelta }f}_{L^p}, &\norm{\Gamma f}_{\Hdot^{1/2}}=\norm{f}_{\Hdot^{1/2}}, &\forall f\in\Hdot^{1/2},\ \forall \Gamma\in\Gcal, 
	\end{array}
\end{equation}
which are needed to apply the method of Section~\ref{sec:method}. We remark that, on the other hand, the property
\begin{equation}\label{eq:theta_invariance}
	\norm{\Re(e^{i t\sqrtDelta} e^{i\theta} f)}_{L^p(\Real^{1+d})}=\norm{\Re(e^{it\sqrtDelta} f)}_{L^p(\Real^{1+d})}
\end{equation}
which is true for $p=4$ (see~\cite{BezRogers}), does not seem to hold for any other value of $p$. Thus, the deficit functional $\psi$ is generally not invariant under the transformation $f\mapsto e^{i\theta} f$. However, we will prove in the forthcoming~\eqref{eq:theta_independence_fstar} that the invariance~\eqref{eq:theta_invariance} does hold for $f=f_\star$; in particular, $f_\star$ is an extremizer if and only if $e^{i\theta}f_\star$ is. Because of this, we have to include this transformation in the definition~\eqref{eq:ComplexM}.

As pointed out in Section~\ref{sec:method}, since $\mathcal G$ is not a compact group, we need the {vanishing at infinity} property~\eqref{eq:vanish_infinity}, which in the present case prescribes that, for each $f, g\in\Hdot^{1/2}$, 
\begin{equation}\label{eq:vanish_infinity_Gcal}
    \begin{array}{ccc}
     \langle \Gamma f| g\rangle_{\Hdot^{1/2}} \to 0 &\text{ if}&\abs{t_0}+\abs{h}+\abs{\log \lambda} + \abs{\operatorname{arctanh}(\alpha)} \to \infty.
     \end{array}
\end{equation}
This property holds true; see~\cite[Lemmas~3.2 and~4.1]{Ramos}. For $d=3$, the propagator $\Re(e^{it\sqrtDelta})$ admits pre-compact maximizing sequences (see Definition~\ref{def:ProfDec}), as proven in the aforementioned~\cite{Ramos}, and a unique maximizer (see Definition~\ref{def:UniqMaximiz}), as proven in~\cite{Fo}. Consequently, Theorem~\ref{thm:main_wave} holds in the stronger, global form. See~\cite{Negro18} for more details. Finally, we record a remark that won't be needed in the sequel but that can be interesting elsewhere.
\begin{remark}\label{rem:Fourier_Maximizers}
	Letting $\Fcal$ denote the Fourier transform, we have 
	\begin{equation}\label{eq:FcalM}
		\Fcal(\MM)=\left\{ \frac{e^{-A\abs\xi + b\cdot \xi + c}}{\abs\xi}\ :\ \Re(A)>0,\, b\in\mathbb{C}^n,\, c\in\mathbb{C}\right\}.
	\end{equation}
%	see Appendix~\ref{app:maxim}.
\end{remark}

\subsection{The Penrose transform}\label{app:penrose}
Recall that
\begin{equation}\label{eq:p_and_nu_cone}
	\begin{array}{ccc}\displaystyle
		p=2\frac{d+1}{d-1}& \text{ and } & \displaystyle \nud=\frac{d-1}{2}.
	\end{array}
\end{equation} 
We consider the sphere $\SSS^d$, with Lebesgue measure $dS$, and we use $X=(X_0, X_1, \ldots, X_d)$ to denote its points; so, $$X_0^2+X_1^2+\ldots+X_d^2=1.$$
We use the following measures on $\TTT:=\Real/(2\pi\Z)$ and on $\SSS^d$ respectively;
\begin{equation}\label{eq:norm_meas}
	\begin{array}{ccc}
		\displaystyle \dTp:=\frac{dT}{C_d}, 
%		&\displaystyle \dTu:=\frac{dT}{2\pi}, 
		&\displaystyle \dSu:=\frac{dS}{\abs{\SSS^d}},
		& \text{where }C_d:=\displaystyle \int_{\TTT}\abs{\cos(\nud T)}^p\, dT.
	\end{array}
\end{equation}
We define a Sobolev norm on the sphere $\SSS^d$ by 
\begin{equation}\label{eq:SphereSobolev}
	\norm{F}_{H^{1/2}(\SSS^d)}^2:=\frac{1}{\nud}\int_{\SSS^d} \sqrt{\nud^2-\DeltaS}(F)\overline{F}\, \dSu.
\end{equation}
With these normalizations, we have for the constant function $1$ that  \begin{equation}\label{eq:normalization}
    \norm{1}_{H^{1/2}(\SSS^d)}=\lVert \Re(e^{iT\sqrt{\nud^2-\Delta_{\SSS^d}}}1)\rVert_{L^p(\TTT\times \SSS^d, \underline{dT}\,\underline{dS})}=1. 
\end{equation}
The following is the main theorem of this subsection.  
\begin{theorem}\label{thm:main_penrose} 
 	There is a surjective $\mathbb C$-linear isometry $$\Ical\colon H^{1/2}(\SSS^d)\to \Hdot^{1/2}(\Real^d)$$ such that
	\begin{enumerate}
		\item[(i)] $\displaystyle \Ical(1)=  f_\star$. 
		\item[(ii)] if $d$ is odd, the deficit functional~\eqref{eq:WaveDeficitComplex} satisfies  $$\displaystyle (\psi\circ \Ical) (F) =\Scal^2\left(\norm{F}_{H^{1/2}(\SSS^d)}^2- \lVert \Re(e^{iT\sqrt{\nud^2-\DeltaS}}F)\rVert_{L^p(\TTT\times \SSS^d, \dTp\,\dSu)}^2\right).$$ 
%		\item[(ii-h)] {for all $d$}, the deficit functional of the half-wave equation satisfies $$\displaystyle (\psih\circ \Ical) (F) =\Scal^2\left(\norm{F}_{H^{1/2}(\SSS^d)}^2- \lVert e^{iT\sqrt{\nud^2-\DeltaS}}F\rVert_{L^p(\TTT\times \SSS^d, \dTu\dSu)}^2\right).$$ 
		\item[(iii)] for each $\theta\in\mathbb R$, the tangent space to $\MM$ satisfies $$\displaystyle\Ical^{-1}(T_{e^{i\theta}f_\star}\MM) =
		%\mathrm{span}_\C\{Y_{0,0}, Y_{1, m}\ :\ m=0,1,\ldots,d\}.
		\left\{ b+ a_0X_0+\ldots+a_d X_d\ |\ b, a_1, \ldots, a_d\in \mathbb C\right\}.$$		
	\end{enumerate}
\end{theorem}
We remark that, as a consequence of~(iii), the tangent space $T_{e^{i\theta}f_\star}\MM$ is in fact independent of $\theta\in\mathbb R$. In the proof of Theorem~\ref{thm:main_penrose} we will also compute the exact value of $\Scal=\norm{\cos(t\sqrtDelta)f_\star}_{L^p}$; 
\begin{equation}\label{eq:ExplScal}
	\Scal=\sqrt{\frac{2}{d-1}} \pi^\frac{1}p \lvert\SSS^d\rvert^{\frac1p-\frac12}.
\end{equation}

\subsubsection{Proof of Theorem~\ref{thm:main_penrose}}
We define the {polar coordinates} on $\Real^d$ as  
\begin{equation}\label{eq:polar_coords}
	\begin{array}{cc}
		x=r\omega, & \text{where }r= |x|\ \text{and}\ \omega\in \SSS^{d-1}, 
	\end{array}
\end{equation}
and the polar coordinates on $\SSS^d$ as 
\begin{equation}\label{eq:polar_coords_sphere}
	\begin{array}{cc}
		(X_0, X_1, \ldots, X_d)=(\cos(R), \sin(R)\, \omega), & \text{where }R\in [0, \pi], \ \omega\in \SSS^{d-1}.
	\end{array}
\end{equation}
We consider the {Penrose map} $\Pcal\colon \Real^{1+d}\to (-\pi, \pi)\times \SSS^d$, defined as
\begin{equation}\label{eq:PenroseMap}
	\Pcal(t, r\omega)=(T, \cos(R), \sin(R)\omega), 
\end{equation}
where 
\begin{equation}\label{eq:PenroseEqns}
	\begin{array}{cc}
		T=\arctan(t+r) + \arctan(t-r), &
		R=\arctan(t+r)-\arctan(t-r).
	\end{array}
\end{equation}
We remark that the restriction of $\Pcal$ to $\{0\}\times \mathbb R^d$, which we denote by $\Pcal|_{t=0}$, is a bijective map onto $\{0\}\times \mathbb S^d\setminus \{(-1,0, \ldots,0)\}$, and it coincides with the usual stereographic projection. We denote\begin{equation}\label{eq:Penrose_notations}
	\begin{array}{ccc}
		\Omega:=\cos T+ \cos R,& \Omega_0:= 1+\cos R, &  \nud:=\frac{d-1}{2}.
	\end{array}
\end{equation}
The functions $\Omega$ and $\Omega_0$ are important here, because they are the conformal factors of $\Pcal$ and $\Pcal|_{t=0}$ respectively; by this we mean that 
\begin{equation}\label{eq:ConformalFactors}
	\begin{split}
		dt^2-dr^2-r^2 d\omega^2&=\Omega^{-2}\left( dT^2- dR^2-\sin^2\!R\, d\omega^2\right),  \\ 
		dr^2+r^2d\omega^2&= \Omega_0^{-2}(dR^2+ \sin^2\!R\,d\omega^2), 
	\end{split}
\end{equation}
where the change of coordinates~\eqref{eq:PenroseEqns} is implicit. The left-hand sides are the metric tensors of the Minkowski space $\mathbb R^{1+d}$ and of the Euclidean space $\mathbb R^d$ respectively, while the terms in brackets in the right-hand sides are the metric tensors of $\mathbb R\times \mathbb S^d$ and of $\mathbb S^d$ respectively; thus $\Pcal$ and $\Pcal|_{t=0}$ are conformal maps. Now, 
a straightforward computation shows that
\begin{equation}\label{eq:OmegaZero_physical}
	\Omega_0\circ \Pcal|_{t=0}= \frac{2}{1+\abs{x}^2}, 
\end{equation}
so, in particular, $f_\star= c 2^{ - \nud }\Omega_0^{\nud} \circ \Pcal|_{t=0}$; this proves~(i). We turn to the proof of~(ii). 
\begin{remark} The image of the Penrose map is
\begin{equation}\label{eq:Pcal_range}
	\Pcal(\Real^{1+d})=\left\{ (T, \cos(R), \sin(R)\omega)\ \mid\ R\in[0, \pi),\ \abs{T}<\pi\right\},
\end{equation}
which is an open submanifold of $[-\pi, \pi]\times \SSS^d$; see Figure~\ref{fig:PenroseTriangle}.
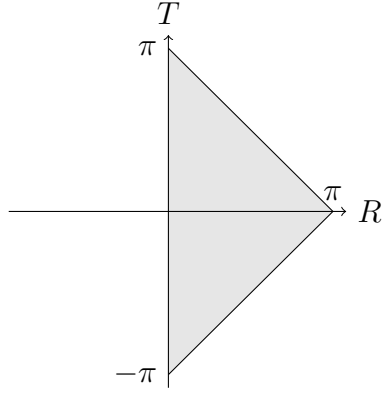
\begin{figure}[!]
\centering
\begin{tikzpicture}[scale=0.7]
\filldraw[fill=gray!20!white] (0, -pi+0.05) 
node[anchor=east] {$-\pi$} 
-- (pi-0.05, 0) 
node[anchor=south]{$\pi$} 
-- (0, pi-0.05) 
node[anchor=east]{$\pi$}
;
\draw[->] (-3, 0) -> (3.34,0) node[anchor=west] {$R$};
\draw[->] (0, -3.34) -> (0, 3.34) node[anchor=south]{$T$};
%\draw[->] (-3, -3) -- (3, 3) node[anchor=south] {$\Xplus$}; 
%\draw[<-] (-3, 3) node[anchor=south] {$\Xminus$} -- (3, -3);
\end{tikzpicture}
\caption{The image of the Penrose map $\Pcal$, in polar coordinates.}
\label{fig:PenroseTriangle}
\end{figure}
\end{remark}

\begin{lemma}[The Penrose transform]\label{lem:PenroseTrans}
	The linear map $\Ical\colon F\mapsto f$, defined by 
	\begin{equation}\label{eq:PenroseTransf}
		f=\frac{1}{\sqrt{\nud\lvert\SSS^d\rvert}}(\Omega_0^{\nud}\, F )\circ \Pcal|_{t=0}
	\end{equation}
	maps $H^{1/2}(\SSS^d)$ onto $\Hdot^{1/2}(\Real^d)$ isometrically, that is 
	\begin{equation}\label{eq:IcalIsometry}
	    \norm{f}_{\Hdot^\frac12(\mathbb R^d)}=\norm{F}_{H^\frac12(\SSS^d)},\qquad \forall f\in \Hdot^\frac12(\mathbb R^d),
	\end{equation}
	and it satisfies 
	\begin{equation}\label{eq:OneSidedPropPullback}
		e^{it\sqrtDelta} f
		= \frac{1}{\sqrt{\nud\lvert\SSS^d\rvert}}\left(
		\Omega^{\nud}\, 
		e^{iT\sqrt{\nud^2-\DeltaS}} F
		\right)\circ \Pcal.
	\end{equation}
%	As a consequence of~\eqref{eq:OneSidedPropPullback}, 
%	\begin{equation}\label{eq:LpPreserved}
%		\norm{e^{it\sqrtDelta} f}_{L^p(\Real^{1+d})}^p=\frac12\norm{e^{iT\sqrtDeltaS}F}_{L^p([-\pi,\pi]\times \SSS^d)}^p.
%	\end{equation}
\end{lemma}
\begin{proof}
Since the conformal factor of $\Pcal|_{t=0}$ is $\Omega_0$, the volume element $dx$ of $\R^d$ and the volume element $dS$ of $\SSS^d$ transform under $\Pcal|_{t=0}$ as follows;
\begin{equation}\label{eq:JacobianPzero}
	dx = \Omega_0^{-d} dS.
\end{equation}
We also have the conformal fractional Laplacian formula of~\cite[eq.~(2)]{Mor02}, which, in our notation, reads
\begin{equation}\label{eq:BrFoMo}
\begin{array}{cc}
	\sqrtDelta (G\circ \Pcal|_{t=0}) = \Omega_0^{\nud+1}\sqrt{\nud^2-\DeltaS}(\Omega_0^{-\nud} G)\circ \Pcal|_{t=0}, & \forall G\in L^2(\mathbb S^d).
\end{array}
\end{equation}
So, if $f$ and $F$ are related by~\eqref{eq:PenroseTransf}, then
\begin{equation}\label{eq:IcalIsometryProof}
	\begin{split}
	\norm{f}_{\dot{H}^{1/2}(\R^d)}^2=\int_{\mathbb R^d} \sqrt{-\Delta}(f)\overline f\, dx&=\frac{1}{\nu_d\abs{\SSS^d}}\int_{\mathbb R^d} \sqrt{-\Delta}((\Omega_0^{\nu_d} F)\circ\Pcal|_{t=0}) (\Omega_0^{\nu_d}\overline{F})\circ\Pcal|_{t=0}\, dx\\ 
	&=\frac{1}{\nu_d\abs{\SSS^d}}\int_{\mathbb S^d}\sqrt{\nu_d^2-\Delta_{\SSS^d}}(F)\overline{F}\, dS \\
	&=\norm{F}_{H^{1/2}(\SSS^d)}^2,
	\end{split}
\end{equation}
which proves~\eqref{eq:IcalIsometry}. 

We now turn to the proof of~\eqref{eq:OneSidedPropPullback}. For notational convenience, we let $C:=\sqrt{\nud\abs{\SSS^d}}$, so that 
$$
Cf=\Omega_0^{\nu_d} F \circ \Pcal|_{t=0}.$$ 
We denote 
\begin{equation}\label{eq:DAlembertians}
	\begin{array}{cc}
	\Box:=\partial_t^2-\Delta,& \Box_{\mathbb S^d} := \partial_T^2 - \Delta_{\SSS^d} - \nu_d^2,
	\end{array}
\end{equation} 
and we let $u:=Ce^{it\sqrtDelta}f$, noting that $\Box u=0$.  Now we define a function on $\Pcal(\R^{1+d})$ by
\begin{equation}\label{eq:CapitalU_Penrose}
	U:=(\Omega^{-\nu_d} u)\circ \Pcal^{-1}.
\end{equation}
The conformality of $\Pcal$ implies the formula
\begin{equation}\label{eq:ConfDAlembert}
	(\Box u)\circ \Pcal^{-1} = \Omega^{\nu_d+2} \Box_{\SSS^d}U ,
\end{equation}
see, e.g.,~\cite[eq.~(A.3.7)]{Horm}; in particular, 
\begin{equation}\label{eq:SphericalWaveEquation}
	\Box_{\SSS^d} U=0,\ \text{on}\ \Pcal(\R^{1+d}).
\end{equation}
To complete the proof of Lemma~\ref{lem:PenroseTrans}, we need to show that $U=e^{iT\sqrt{\nu^2-\Delta_{\SSS^d}}}F$. Since both functions solve the PDE~\eqref{eq:SphericalWaveEquation}, it will be enough to prove that
\begin{equation}\label{eq:ToProveInitialData}
	\begin{array}{cc}
		U|_{T=0}=F, & \partial_T U|_{T=0} = i\sqrt{\nud^2-\Delta_{\SSS^d}} F.
	\end{array}
\end{equation}
The first of these identities is manifestly true. To compute $\partial_T U|_{T=0}$, we begin by differentiating~\eqref{eq:PenroseEqns}, yielding
\begin{equation}\label{eq:partial_t_at_zero}
	\left.\frac{\partial}{\partial t}\right|_{t=0} = \frac{2}{1+r^2} \left.\frac{\partial}{\partial T}\right|_{T=0} = \Omega_0\left.\frac{\partial}{\partial T}\right|_{T=0}.
\end{equation}
By the definition of $u$, we have that $\partial_t u|_{t=0}= Ci\sqrt{-\Delta}f$. So, by~\eqref{eq:partial_t_at_zero} and~\eqref{eq:BrFoMo},
\begin{equation}\label{eq:partial_T_U}
		\partial_T U|_{T=0}=\Omega_0^{-1-\nu_d}C i \sqrt{-\Delta}f\circ\Pcal|_{t=0}^{-1} =i\sqrt{\nud^2-\Delta_{\SSS^d}}F,
\end{equation}
which concludes the proof of~\eqref{eq:ToProveInitialData} and of Lemma~\ref{lem:PenroseTrans}.

%We have thus proved that the function $U$ defined by
%\begin{equation}\label{eq:StandardPenrose}
%	U= \frac{\Omega^{-\nud} u}{\sqrt{\nu_d \abs{\SSS^d}} \circ \Pcal^{-1}
%	%u=\frac{1}{\sqrt{\nud\abs{\SSS^d}}}\Omega^{\nud} U\circ \Pcal,
%\end{equation}
%satisfies the initial value problem
%\begin{equation}\label{eq:spherical_IC}
%	\begin{cases} 
%		\partial_T^2 U= (\Delta_{\SSS^d} -\nud^2) U, & \text{ on }\Pcal(\R^{1+d}), \\ 		
%		\left.U\right|_{T=0}=\Omega_0^{-\nud} f\circ \Pcal|_{t=0}^{-1}, \\ 
%		\left.\partial_T U\right|_{T=0}=i\Omega^{-\nud-1} \sqrtDelta f\circ \Pcal|_{t=0}^{-1}.
%	\end{cases}
%\end{equation}
%By~\eqref{eq:BrFoMo} we see that 
%\begin{equation}\label{eq:Df_is_DSF}
%  	\sqrtDelta f= \Omega_0^{\nud+1}\sqrt{\nud^2-\DeltaS}F,
%\end{equation}
%so $\left.\partial_T U\right|_{T=0}=i\sqrt{\nud^2-\Delta}U_0$. We conclude that 
%\begin{equation}\label{eq:U_is_spherical_half_wave}
%	U=e^{i\sqrt{\nud^2-\Delta}T}F,
%\end{equation}
%as claimed.
\end{proof}
Now we discuss integration. Similarly to what we observed in the proof of the previous theorem, the volume element $dtdx$ of $\R^{1+d}$ and the volume element $dTdS$ of $\R\times \SSS^d$ transform under the conformal map $\Pcal$ as $$dtdx=\Omega^{-(d+1)}dTdS.$$   
Thus, for each $u\colon\mathbb{R}^{1+d}\to\mathbb C$, 
	\begin{equation}\label{eq:Penrose_change_variable}
	    \begin{split}
		\iint_{\Real^{1+d}} u\, dtdx&= \iint_{\Pcal(\Real^{1+d})} (u\circ \Pcal)\Omega^{-(d+1)}\, dTdS\\ 
% 		&=C(d) \lvert \SSS^d\rvert \iint_{\Pcal(\Real^{1+d})} (u\circ \Pcal)\Omega^{-(d+1)}\, \dTu\dSu.
	    \end{split}
	\end{equation}
The domain of integration in the right-hand side of~\eqref{eq:Penrose_change_variable} is the triangular region $\Pcal(\Real^{1+d})$ (see Figure~\ref{fig:PenroseTriangle}), which is not a Cartesian product, and this is a nuisance. We will see that we can in fact integrate on a more symmetric region.  

We now introduce the spherical harmonics. For each $\ell\in \N_{\ge 0}$, we let $N_d(\ell)$ denote the number of spherical harmonics of degree $\ell$ on the sphere $\SSS^d$. We fix, once and for all, a complete orthonormal system of $L^2(\SSS^d, \dSu)$ 
\begin{equation}\label{eq:SphHarm}
	\left\{ Y_{\ell, m}\ :\ \ell\in \N_{\ge 0},\ m=1, \ldots, N_d(\ell)\right\},
\end{equation}
with the property that
\begin{equation}\label{eq:SphHarmEigen}
	-\DeltaS Y_{\ell,m}= \ell(\ell+2\nud)Y_{\ell, m}.
\end{equation}
We remark that 
\begin{equation}\label{eq:DEss}
	\sqrt{\nud^2-\DeltaS} Y_{\ell, m}=\sqrt{\nud^2+\ell^2+2\ell\nud}Y_{\ell, m}=(\ell+\nud)Y_{\ell, m}.
\end{equation}
In particular,
\begin{equation}\label{eq:SphericalSobolev}
	\norm{ F}_{H^{1/2}(\SSS^d)}^2= \sum_{\ell\ge 0}\sum_{m=1}^{N_d(\ell)} (\ell+\nud)\lvert\Fhat(\ell, m)\rvert^2,
\end{equation}
where 
\begin{equation}\label{eq:sphericalFourier}
	\Fhat(\ell, m):=\int_{\SSS^d} F(X)Y_{\ell, m}(X)\, \dSu, 
\end{equation}
Finally, as it is well-known, $Y_{\ell, m}(X)$ is a homogeneous harmonic polynomial of degree $\ell$. Thus, in particular, 
\begin{equation}\label{eq:SphHarmSym}
	Y_{\ell, m}(-X)=(-1)^\ell Y_{\ell, m}(X), \qquad \forall X\in \SSS^d.
\end{equation}
We can now prove the point~(ii).
\begin{lemma}\label{lem:SymTrick}
	Let $p=2\frac{d+1}{d-1}$, $F\in H^{1/2}(\SSS^d)$ and $f=\Ical(F)$. For each integer $d\ge 2$, 
%	\begin{equation}\label{eq:IcalIsometry}
%		\norm{f}_{\Hdot^{1/2}}=\norm{F}_{H^{1/2}(\SSS^d)}, 
%	\end{equation}
%	and 
	\begin{equation}\label{eq:sym_trick}
		\iint_{\Real^{1+d}} \abs{e^{it\sqrtDelta} f}^p\, dtdx = \frac12C \int_{-\pi}^\pi\int_{\SSS^d} \abs{e^{iT\sqrt{\nud^2-\DeltaS}} F}^p\, \dTu \dSu,
	\end{equation}
	and each odd integer $d\ge 3$, 
	\begin{equation}\label{eq:sym_trick_full_wave}
		\iint_{\Real^{1+d}} \abs{\Re(e^{it\sqrtDelta} f)}^p\, dtdx = \frac12 C\int_{-\pi}^\pi\int_{\SSS^d} \abs{\Re( e^{iT\sqrt{\nud^2-\DeltaS}} F)}^p\,\dTu \dSu.
	\end{equation}
	where $C=C(d)\nud^{-\frac{p}2}\lvert\SSS^d\rvert^{1-\frac{p}2}$.
\end{lemma}
\begin{remark}\label{rem:Half_Penrose_Novelty}
	Only the formula~\eqref{eq:sym_trick_full_wave} is needed for the proof of Theorem~\ref{thm:main_penrose}. However,~\eqref{eq:sym_trick} will allow us to discuss the case of the half-wave propagator, mentioned in the Introduction; see the forthcoming Remark~\ref{rem:half:wave}.
\end{remark}
\begin{proof}
Combining Lemma~\ref{lem:PenroseTrans}, and the integration formula~\eqref{eq:Penrose_change_variable}, we get
\begin{equation}\label{eq:HalfWaveTriangle}
\begin{split}
	\iint_{\Real^{1+d}} \abs{e^{it\sqrtDelta}f}^p\,dtdx
	&=
	2C\iint_{\Pcal(\Real^{1+d})} \abs{e^{iT\sqrt{\nud^2-\DeltaS}} F}^p\Omega^{p\nud-(d+1)}\, \dTu\dSu 
	\\ &=2C\iint_{\Pcal(\Real^{1+d})} \abs{e^{iT\sqrt{\nud^2-\DeltaS}} F}^p\, \dTu\dSu,
\end{split}
\end{equation}
where we used that $p\nud=d+1$. Now, by~\eqref{eq:DEss},
\begin{equation}\label{eq:spherical_wave}
	e^{iT\sqrt{\nud^2-\DeltaS}} F =\sum_{\ell\ge 0}\sum_{m=1}^{N(\ell)} e^{iT(\ell+\nud)}\Fhat(\ell, m)Y_{\ell, m},
\end{equation}
and so the following crucial simplification occurs;
\begin{equation}\label{eq:u_no_nu}
	\abs{e^{iT\sqrt{\nud^2-\DeltaS}}F}=\abs{\sum_{\ell\ge 0}\sum_{m=1}^{N(\ell)} e^{iT\ell} \Fhat(\ell, m)Y_{\ell, m}}.
\end{equation}
Thus, in particular, $U(T, X):=\abs{e^{iT\sqrt{\nud^2-\DeltaS}}F}^p$ satisfies
\begin{equation}\label{eq:key_symmetry}
	\begin{array}{cc}
	U(T+2\pi, X)=U(T, X),&	U(T+\pi, -X)=U(T, X),
	\end{array}
\end{equation}
for all $(T, X)\in \Real\times \SSS^d$. As observed in~\cite{Negro18},  the symmetries~\eqref{eq:key_symmetry} imply  
\begin{equation}\label{eq:HalfWaveRectangle}
	\int_{-\pi}^\pi \int_{\SSS^d} U(T, X)\, \dTu\dSu = 2\iint_{\Pcal(\Real^{1+d})} U(T, X)\, \dTu\dSu.
\end{equation}
Applying this to~\eqref{eq:HalfWaveTriangle}, we complete the proof of~\eqref{eq:sym_trick}. 

To prove~\eqref{eq:sym_trick_full_wave} it suffices to show that $$U(T, X):=\abs{\Re(e^{iT\sqrt{\nud^2-\DeltaS}}F)}^p$$ satisfies the symmetries~\eqref{eq:key_symmetry}. This is true only if $d$ is an odd integer; see the aforementioned~\cite{Negro18}.
\end{proof}
One immediate consequence of the previous lemma is that $\lVert \Re(e^{it\sqrtDelta}e^{i\theta}f_\star)\rVert_{L^p}$ is independent of $\theta\in\mathbb R$. Indeed,
\begin{equation}\label{eq:theta_independence_fstar}
	\lVert \Re(e^{it\sqrtDelta} e^{i\theta}f_\star) \rVert_{L^p}^p= \frac12 C\int_{-\pi}^\pi \int_{\SSS^d} \lvert \Re e^{i(T\nud+\theta)}\rvert^p\, \dTu\dSu;
\end{equation}
and the last integral is independent of $\theta$.
The previous lemma and Lemma~\ref{lem:PenroseTrans} immediately imply point~(ii) of Theorem~\ref{thm:main_penrose}; we remark that~(ii) is \textbf{false} if $d$ is an even integer.

We turn to a sketchy proof of the remaining point~(iii). We recall that 
\begin{equation}\label{eq:MManifold}
	\MM=\{ r\Gamma(e^{i\theta}f_\star)\ :\ r\ge 0, \theta \in \Real, \Gamma\in\Gcal\},
	\end{equation}
and that the tangent space at $e^{i\theta}f_\star$ is 
\begin{equation}\label{eq:tangent_space_recall}
	T_{e^{i\theta}f_\star} \MM=\mathrm{span}_{\r}\left( f_\star, if_\star, i\sqrtDelta e^{i\theta}f_\star, \partial_{x_j} e^{i\theta}f_\star, \left(\nud+x\cdot\nabla\right)e^{i\theta}f_\star, ix_j\sqrtDelta e^{i\theta}f_\star\right),
\end{equation}
%where $\Gcal$ is the unitary group on $\Hdot^{1/2}(\Real^d)$ generated by the following transformations: 
%\begin{equation}\label{eq:GcalListTransformation}
%	\begin{array}{cc}
%		f\mapsto e^{i t_0\sqrtDelta} f, & t_0\in \Real, \\ 
%		f\mapsto f(\cdot+ h), & h\in \Real^d, \\
%		f\mapsto e^{\nud\lambda}f(e^{\lambda}\cdot), & \lambda\in \Real, \\
%		f\mapsto f(R\cdot), & R\in SO(d), \\
%		f\mapsto L^\alpha(e^{i(\cdot) \sqrtDelta}f)|_{t=0}, & \alpha\in(-1,1). \\
%	\end{array}
%\end{equation}
%Thus, $\MM$ is a finite dimensional smooth manifold. By definition, its tangent space $T_{f_\star} \MM$ is the real vector space spanned by differentiating $r\Gamma(e^{i\theta}f_\star)$ with respect to $r, \theta$ and the parameters in~\eqref{eq:GcalListTransformation}; that is,
where $j=1, 2, \ldots, d$. All these derivatives can be explicitly computed, and recalling that $f_\star=C2^{\nud}\Ical(1+X_0)$, the computation is elementary, except for the use of~\eqref{eq:BrFoMo}, which yields $$\sqrtDelta(\Omega_0^{\nud})=\nud\Omega_0^{\nud+1}.$$
The result is 
\begin{equation}\label{eq:tangent_space_penrose}
	\Ical^{-1}(T_{e^{i\theta}f_\star}\MM)=\mathrm{span}_\Real\left(1, i, X_0, iX_0, \ldots, X_d, iX_d\right),
\end{equation}
which completes the proof of the point~(iii). Theorem~\ref{thm:main_penrose} is now completely proven.

\subsection{Proof of Theorem~\ref{thm:main_wave}}

We consider the functional 
\begin{equation}\label{eq:psi_tilde}
	\psitilde(F):=\frac{(\psi\circ \Ical)(F)}{\Scal^2},
\end{equation} 
where $\Ical$ is the Penrose isometry of $H^{1/2}(\SSS^d)$ onto $\Hdot^{1/2}(\Real^d)$ defined in Lemma~\ref{lem:PenroseTrans}; thus, as we proved in the previous subsection,
\begin{equation*}%\label{eq:psi_tilde_explicit}
	\psitilde(F)=\norm{F}_{H^{1/2}(\SSS^d)}^2-\lVert \Re \left(e^{iT\sqrt{\nud^2-\DeltaS}}F\right) \rVert_{L^p(\dTp\dSu)}^2,
\end{equation*}
where we recall that $\dTp$ and $\dSu$ are the Lebesgue measures on $\mathbb{T}$ and $\mathbb S^d$ with the normalization
\begin{equation}\label{eq:LebNorm}
	\begin{array}{cc}
		\displaystyle \dTp:=\frac{dT}{\int_{\TTT}\abs{\cos(\nud T')}^p\, dT'}, &\displaystyle \dSu:=\frac{dS}{\abs{\SSS^d}},
	\end{array}
\end{equation}
 and the norm on the Sobolev space $H^{1/2}(\SSS^d)$ is 
\begin{equation}\label{eq:SobolevSphereMainText}
	\norm{F}_{H^{1/2}(\SSS^d)}^2:=\frac{1}{\nud}\int_{\SSS^d}\sqrt{\nud^2-\Delta}(F)\overline{F}\, \dSu.
\end{equation}

In order to apply Theorem \ref{thm:general} and Remark~\ref{rem:RealLinear}, where now $f_\star$ has to be interpreted as the constant function $1$ on $\SSS^d$, we need to verify that $\tilde \psi'(e^{i\theta})(F)=0$ for all $F\in H^{1/2}(\s^d)$ and all $\theta\in\mathbb R$. This is equivalent to showing that there is $\mu\in \r$ such that
\begin{equation}\label{eq:ELSphere_Wave}
	\int_{-\pi}^\pi \int_{\SSS^d} \abs{\Re (e^{i\theta}e^{i\nud T})}^{p-2}\Re(e^{i\theta}e^{i\nud T})\Re(e^{i\sqrt{\nud^2-\DeltaS}T}F)\, \dTp\dSu =\mu \Re\Braket{ F | 1}_{H^{\frac12}},
\end{equation}
for all $F\in H^{1/2}(\SSS^d)$. Now, expanding in spherical harmonics,
$$e^{iT\sqrt{\nud^2-\DeltaS}}F=\sum_{\ell\ge 0}\sum_{m=1}^{N_d(\ell)} e^{iT(\ell+\nud)}\Fhat(\ell, m) Y_{\ell, m}, $$
so, by using the $L^2(\SSS^d)$ orthonormality of $Y_{\ell, m}$ we see that the left-hand side of~\eqref{eq:ELSphere_Wave} reduces to 
$$\left( \int_{-\pi}^\pi\abs{\cos(\nud T+\theta)}^{p-2}\cos(\nud T+\theta) e^{i\nud T}\, \dTp\right)\Re \Fhat(0,0), $$ 
while the right-hand side is $\Re \Fhat(0,0)$. Thus~\eqref{eq:ELSphere_Wave} is satisfied.

We now turn to the verification of the second-order condition~\eqref{eq:real_positive_tangent} of Remark~\ref{rem:RealLinear}. Using Appendix \ref{app:abstvar}, this amounts to the proof that the quadratic form 
\begin{equation}\label{eq:Q_Wave}
	Q(F, F)=2(p-1)\int_{-\pi}^\pi\int_{\SSS^d} \abs{\cos (\nud T+\theta)}^{p-2}\abs{\sum_{\ell\ge 2}\sum_{m=0}^{N_d(\ell)}\Re [e^{iT(\ell+\nud)}\Fhat(\ell, m)]Y_{\ell, m}}^2\, \dTp\dSu,
\end{equation}
satisfies, for some $\rho>0$, the following bound, uniform in $\theta\in\mathbb R$,
\begin{equation}\label{eq:QCoercivity}
	\begin{array}{cc}
	Q(F, F)\le (2-\rho)\norm{F}_{H^{1/2}(\SSS^d)}^2.
	\end{array}
\end{equation}
Here $F$ is such that, letting $f=\Ical(F)$, we have the orthogonality $f\in (T_{e^{i\theta}} M)^\bot$. By the point~(iii) of Lemma~\ref{lem:PenroseTrans}, and recalling that the polynomials of first degree are precisely the spherical harmonics of degree $1$, we see that this is equivalent to 
\begin{equation}\label{eq:ortho_spherical_harmonics}
	\begin{array}{cc}
	\hat{F}(\ell, m)=0,& \text{ for }\ell=0, 1.
	\end{array}
\end{equation}

To prove~\eqref{eq:QCoercivity}, we use the change of variable $T\mapsto T-\theta/\nud$, we let $\phi:=\frac{\theta}{\nud}+\nud$, and we use again the $L^2(\SSS^d)$ orthonormality of $Y_{\ell, m}$ to obtain 
\begin{equation}\label{eq:Q_Wave_Two}
	2(p-1)\sum_{\ell\ge 2}\sum_{m=0}^{N_d(\ell)}
	\int_{-\pi}^\pi\abs{\cos (\nud T)}^{p-2}[\Re e^{i\phi}e^{iT\ell}\Fhat(\ell, m)]^2\, \dTp
\end{equation}
Up to replacing $F\mapsto e^{-i\phi}F$, we can assume that $\phi=0$ without loss.

\subsubsection{The $d=3$ case.} In this case, $p=4$ and $\nu_3=1$, so~\eqref{eq:Q_Wave_Two} reduces to 
\begin{equation}\label{eq:ThreedQ}
	\dfrac{\displaystyle6\sum_{\ell\ge 2}\sum_{m=0}^{N_3(\ell)} \int_{-\pi}^\pi \left[ \Re( \cos T e^{iT\ell} \Fhat(\ell, m))\right]^2\, dT}{\int_{-\pi}^\pi (\cos T)^4\, dT}=4\sum_{\ell\ge 2}\sum_{m=0}^{N_3(\ell)}\abs{\Fhat(\ell, m)}^2,
\end{equation}
And so we conclude that 
\begin{equation}\label{eq:ThreedEnd}
	\frac{Q(F, F)}{\norm{F}_{H^{1/2}(\SSS^3)}^2}= 4\frac{\sum_{\ell\ge 2}\lvert{\Fhat(\ell, m)\rvert}^2}{\sum_{\ell\ge 2}(\ell+1)\lvert{\Fhat(\ell, m)\rvert}^2}\le \frac43, 
\end{equation}
thus,~\eqref{eq:QCoercivity} is satisfied with $\rho=\frac23$. By Theorem~\ref{thm:general}, we can therefore conclude that there are $C>0$ and $\delta\in(0, 1)$ such that 
\begin{equation}\label{eq:ThreeDConclusion}
    \psi(f)\ge C\dist(f, \MM)^2,
\end{equation}
provided that 
\begin{equation}\label{eq:removable_distance_condition}
    \dist(f, \MM)<\delta\norm{f}_{\Hdot^{1/2}(\mathbb R^3)}.
\end{equation}
As stated in the Introduction, the inequality~\eqref{eq:ThreeDConclusion} actually holds even without this condition, up to possibly replacing $C$ with a smaller constant; this follows from the general Theorem~\ref{thm:general_sharpened}, and has been already observed in~\cite{Negro18}.

\subsubsection{The general case $d\ge 5$.} We need the following Fourier series
\begin{equation}\label{eq:cosine_fourier_series}
	\abs{\cos(\beta T)}^{2\alpha}=
	\frac{\Gamma(\alpha+\frac12)}{\sqrt{\pi}\Gamma(\alpha+1)}+\frac{2}{\sqrt\pi}\sum_{h=1}^{\infty}\binom{\alpha}{h}\frac{\Gamma(h+1)\Gamma(\alpha+\frac12)}{\Gamma(\alpha+h+1)}\cos(2h\beta T).
\end{equation}
for all $\alpha>0$ and $\beta\in\Real$.
%Sketch: start from the absolutely converging sum 
%\begin{equation}\label{eq:Raabe}
%	\abs{u}^{2\alpha}=\sum_{j=0}^\infty\binom{\alpha}{j}(u^2-1)^j,\qquad \forall u\in[-\sqrt2, \sqrt 2];
%\end{equation}
%the fact that this sum converges absolutely is a consequence of the criterion of Raabe. Let $u=\cos(\beta T)$, so that 
%\begin{equation}\label{eq:cos_first_sum}
%	\abs{\cos(\beta T)}^{2\alpha} = \sum_{j=0}^\infty \binom{\alpha}{j}(-1)^j(\sin \beta T)^{2j}.
%\end{equation}
%Now develop, using $\sin(\beta T)=\frac{1}{2i}(e^{i\beta T}-e^{-i\beta T})$,
%\begin{equation}\label{eq:sin_develop}
%	(\sin \beta T)^{2j}=\frac{1}{2^{2j}}\left[\binom{2j}{j}+2\sum_{m=1}^j (-1)^m\binom{2j}{j-m}\cos(2m\beta T)\right],
%\end{equation}
%We conclude that $\abs{\cos(\beta T)}^{2\alpha}$ equals the series 
%\begin{equation}\label{eq:cosine_hypergeom}
%	\sum_{j=0}^\infty \frac{(-1)^j}{2^{2j}}\binom{\alpha}{j}\binom{2j}{j}+\sum_{j=0}^\infty\sum_{m=1}^j\frac{(-1)^{j+m}}{2^{2j-1}}\binom{\alpha}{j}\binom{2j}{j-m}\cos(2m\beta T).
%\end{equation}
%Inverting the order of summation of the second series, and feeding it to Maple, produces~\eqref{eq:cosine_fourier_series}.
Recall that $\dTp=\frac{dT}{\int_{-\pi}^\pi\abs{\cos(\nud T')}^p\, dT'}$. Thus, letting $\alpha=p/2, \beta=\nud$, we compute
\begin{equation}\label{eq:dT_normalized}
	\dTp=\frac{\Gamma(\frac{p+2}{2})}{2\sqrt\pi \Gamma(\frac{p+1}{2})}\, dT,
\end{equation}
while letting $\alpha=(p-2)/2$ and $\beta=\nud$, we compute
\begin{equation}\label{eq:CosineFourier}
	\abs{\cos(\nud T)}^{p-2}=a_0+\sum_{h=1}^\infty a_h\cos(2h\nud T),
\end{equation}
where the coefficients are given by
\begin{equation}\label{eq:CosineCoeffs}
	\begin{array}{cc}
	\displaystyle	a_0:=\frac{\Gamma(\frac{p-1}{2})}{\sqrt\pi \Gamma(\frac{p}2)}, &\displaystyle a_h:=\frac{2}{\sqrt\pi}\frac{\Gamma(\frac{p}{2})\Gamma(\frac{p-1}{2})}{\Gamma(\frac{p}{2}+h)\Gamma(\frac{p}{2}-h)},\ h\ge 1;
	\end{array}
\end{equation}

We will also need the formulas 
\begin{equation}\label{eq:Orthog_Formulas}
	\begin{array}{c}
		\displaystyle \int_{-\pi}^\pi (\Re(ze^{iT\ell}))^2\, \dTp=\frac{\sqrt\pi\Gamma(\frac{p+2}{2})}{2\Gamma(\frac{p+1}{2})}\abs z^2, \\ \displaystyle \int_{-\pi}^\pi \cos(2h\nud T)(\Re(ze^{iT\ell}))^2\, \dTp=\frac{\sqrt\pi\Gamma(\frac{p+2}{2})}{4 \Gamma(\frac{p+1}{2})}\Re(z^2)\delta_{\ell, h\nud}.
	\end{array}
\end{equation}
Using all of this, we obtain from~\eqref{eq:Q_Wave_Two} (we omit the sums in $m$ from now on)
\begin{equation}\label{eq:Q_explicit}
	\begin{split}
	Q(F, F)&=\frac{(p-1)\sqrt\pi \Gamma(\frac{p+2}{2})}{\Gamma(\frac{p+1}{2})}\left[ a_0\sum_{\ell\ge 2}\abs{\Fhat(\ell,m)}^2+\frac12\sum_{h=1}^\infty a_h\Re(\Fhat(h\nud,m)^2)\right]\\
	&=\frac{(p-1)\Gamma(\frac{p+2}{2})\Gamma(\frac{p-1}{2})}{\Gamma(\frac{p+1}{2})\Gamma(\frac{p}2)}\left[\sum_{\ell\ge 2}\abs{\Fhat(\ell, m)}^2+\sum_{h=1}^\infty \frac{\Gamma(\frac{p}2)^2\Re[\Fhat(h\nud,m)^2]}{\Gamma(\frac{p}2+h)\Gamma(\frac{p}2-h)}\right],
	\end{split}
\end{equation}
and now we notice, using the property $\Gamma(\beta+1)=\beta\Gamma(\beta)$, that 
\begin{equation}\label{eq:C_Realwave}
	\frac{(p-1)\Gamma(\frac{p+2}{2})\Gamma(\frac{p-1}{2})}{\Gamma(\frac{p+1}{2})\Gamma(\frac{p}2)}=p.
\end{equation}
We thus have the inequality
% \footnote{This is sharp only for $d=5$, and it is not sharp for $d\ge 7$. For $d=7$, a more careful argument will give $\rho=\frac{2}{\nud+2}$, which is bigger than~\eqref{eq:wave_rho} and is the same constant of the half-wave case; see later. I will write a more careful argument, with the sharp $\rho$, in the final version.}
\begin{equation}\label{eq:SharpLtwo}
	Q(F, F)\le C \sum_{\ell \ge 2} \lvert{\Fhat(\ell, m)\rvert}^2, 
\end{equation}
where 
\begin{equation}\label{eq:C_Sharp}
	C=\max_{h \ge 1} \left(p+ p\abs{\frac{\Gamma(\frac{p}2)^2}{\Gamma(\frac p2+h)\Gamma(\frac p2-h)}}\right).
\end{equation}
We claim that the maximum is attained at $h=1$, so  
\begin{equation}\label{eq:C_h_one}
	C=p+\frac{p-2}{p}.
\end{equation}
Once this is proven, the proof of Step 2 will be completed by observing that 
\begin{equation*}
	\frac{Q(F, F)}{\norm{F}^2_{H^{1/2}}} \le \frac{C\sum_{\ell\ge 2}\lvert{\Fhat(\ell)\rvert}^2}{\frac{1}{\nud}\sum_{\ell\ge 2} (\ell+\nud)\lvert{\Fhat(\ell)\rvert}^2}\le \frac{2(\nud+1)+\frac{\nud}{\nud+1}}{2+\nud}=\frac{2\nud+1}{\nud+1},
\end{equation*}
and $\frac{2\nud+1}{\nud+1}=2-\frac1{\nud+1}$, so~\eqref{eq:QCoercivity} holds with 
\begin{equation}\label{eq:wave_rho}
	\rho=\frac1{\nud+1}.
\end{equation}

To prove that the maximum in~\eqref{eq:C_Sharp} is attained at $h=1$, we introduce the notation
\begin{equation}\label{eq:alpha_var}
	\begin{array}{cc}
		\alpha=\frac{p}2, & \alpha\in (1, \frac32],
	\end{array}
\end{equation}
and we get rid of the Gamma functions with negative arguments via the following computation, in which $(a)_n$ denotes the rising factorial;
\begin{equation}\label{eq:GammaManipulation}
	\frac{\Gamma(\alpha)}{\Gamma(\alpha-h)}=(\alpha-h)_h=(-1)^h(2-\alpha)_{h-1}(\alpha-1)=
	(-1)^{h}(\alpha-1)\frac{\Gamma(h-\alpha+1)}{\Gamma(2-\alpha)},
\end{equation}
so we conclude that 
\begin{equation}\label{eq:Abs_Max}
	\abs{\frac{\Gamma(\alpha)^2}{\Gamma(\alpha+h)\Gamma(\alpha-h)}}= \frac{\Gamma(\alpha)(\alpha-1)}{\Gamma(2-\alpha)}\frac{\Gamma(h-\alpha+1)}{\Gamma(\alpha+h)}.
\end{equation}
We notice now that the function 
\begin{equation}\label{eq:auxiliary_fct}
	g(\alpha, h):=\frac{\Gamma(h-\alpha+1)}{\Gamma(\alpha+h)}
\end{equation}
is decreasing in $h\ge 1$ for each fixed $\alpha>1/2$, thus in particular it is decreasing for fixed $\alpha\in (1, \frac32]$; indeed,
\begin{equation}\label{eq:log_derivative}
	\frac{\partial}{\partial h} \log g(\alpha, h)= (\log\Gamma)'(h-\alpha+1)-(\log\Gamma)'(\alpha +h) \le 0, 
\end{equation}
because the derivative $(\log\Gamma)'$ is an increasing function, since $\Gamma$ is log-convex. This proves that the maximum in~\eqref{eq:C_Sharp} is attained at $h=1$ and concludes the proof.

\begin{remark}\label{rem:half:wave}
As announced in the introduction, Theorem~\ref{thm:main_wave} holds for the functional 
 \begin{equation}\label{eq:half_wave_deficit_again}
    \begin{array}{cc}
\displaystyle     \psi_h(f):=\mathcal{C}_h^2 \norm{f}_{\Hdot^\frac12(\mathbb R^d)}^2-\norm{e^{it\sqrtDelta} f}_{L^p}^2, & \text{ where } \mathcal{C}_h:=\norm{e^{it\sqrtDelta} f_\star}_{L^p(\mathbb R^{1+d}},
\end{array}
\end{equation}
for each $d\ge 2$, regardless of its parity. Via a Penrose transform, the proof amounts to studying the functional
\begin{equation}\label{eq:psi_tilde_explicit}
	\psitilde_h(F)=\norm{F}_{H^{1/2}(\SSS^d)}^2-\lVert e^{iT\sqrt{\nud^2-\DeltaS}}F \rVert_{L^p(\dTp\dSu)}^2,
\end{equation}
with the normalization $$\dTp \dSu=\frac{dTdS}{2\pi \lvert \SSS^d\rvert}.$$
This can be done with computations analogous to, but simpler than the ones of the present section.
\end{remark}

\section{Paraboloid adjoint Fourier restriction - Schrödinger equation}
We now turn to the proof of Theorem~\ref{thm:main_schro}. The functional that we will study here is 
\begin{equation}\label{eq:SchroDeficitRecall}
    \begin{array}{cc}
    \phi(f):=\mathcal{A}_d^2 \norm{f}_{L^2(\mathbb R^d)}^2 - \norm{e^{it\Delta}f}_{L^p(\mathbb R^{d+1})}^2, & p:=2+\frac4d,
    \end{array}
\end{equation}
where $\mathcal{A}_d=4^{-\frac{d}{8+4d}}\left(1+\frac{2}{d}\right)^{-\frac{d^2}{8+4d}}$ and $e^{it\Delta}$ is the Schr\"odinger propagator. We want to apply the method of Section~\ref{sec:method}. To this end, we note that the manifold of Gaussians 
$$
\GG=\{e^{a|x|^2 +b\cdot x + c}: a,c\in \cp, b\in \cp^d \text{ and } \Re \, a <0\}\subset L^2(\mathbb R^d),
$$
can be written in the form 
\begin{equation}\label{eq:Gaussians_are_M}
    \GG=\left\{ z\Gamma(f_\star)\ :\ z\in\mathbb C\setminus \{0\}\, , \Gamma\in \T\right\}, 
\end{equation}
where  $\T$ denotes the group of transformations of $L^2(\mathbb R^d)$ generated by:
\begin{itemize}
\item Space translations $f(x)\mapsto f(x-x_0)$, $x_0\in \r^d$;
\item Frequency translations $f(x)\mapsto e^{ib\cdot x}f(x)$, $b\in \r^d$;
\item Finite time propagation (or time translation) $f(x)\mapsto e^{it_1\Delta}(f)(x)$;
\item Rotations $f(x)\mapsto f(Rx)$, $R\in SO(d)$;
\item Scaling $f(x)\mapsto \la^{-d/2}f(\la x)$,  $\la>0$.
\end{itemize}

Using this symmetry group one can show that
\begin{equation}\label{eq:tangGG}
T_{e^{-\pi |\cdot|^2}} \GG= {\rm span}_{\cp}\{e^{-\pi |x|^2},x_1e^{-\pi |x|^2},x_2e^{-\pi |x|^2},...,x_de^{-\pi |x|^2},|x|^2e^{-\pi |x|^2}\}.
\end{equation}
It is also not difficult to show the vanishing at infinity property~\eqref{eq:vanish_infinity} of $\T$; see, for example, B\'egout and Vargas~\cite{BeVa}. For $d=1, 2$, Theorem~\ref{thm:main_schro} holds in its stronger, global form, as a consequence of the general Theorem~\ref{thm:general_sharpened}. The necessary pre-compactness of maximizing sequences, as in Definition~\ref{def:ProfDec}, has been proven in the aforementioned~\cite{BeVa}, while the uniqueness up to symmetries of the maximizer is due to Foschi~\cite{Fo}.

\subsection{Lens transform} \label{sec:lens}
To define and also give context to the Lens transform we will need to use the Hermite and Laguerre orthogonal functions of $L^2(\r^d)$ and for this we follow the set up of \cite[Sections 2.1, 2.2 and 3.1]{Gon}. For a given vector $\bn \in \z_+^d$ we let 
$$
F_\bn(x) = H_{n_1}(\sqrt{4\pi}\,x_1)...H_{n_d}(\sqrt{4\pi}\,x_d)e^{-\pi |x|^2},
$$
where $H_n(z)$ are the monic Hermite polynomials orthogonal with respect to the Gaussian normal distribution (above $|\cdot|$ stands for the euclidean norm). The functions $\{F_{\bn}\}_{\bn\in\z_+^d}$ form an orthogonal basis of $L^2(\r^d)$ and are eigenfunctions of the Fourier transform
$$
\ft F_{\bn}(\xi) = \int_{\r^d} F_{\bn}(x)e^{-2\pi i x \cdot \xi}dx=(-i)^{|\bn|_1} F_{\bn}(\xi),
$$
where $|\bn|_1=n_1+n_2+...+n_d$.  Similarly, for a given parameter $\nu>-1$ we let $\{L^\nu_m(r)\}_{m\geq 0}$ be the generalized Laguerre polynomials. These are the orthogonal polynomials with respect to the measure 
$$
d\mu(r)=\frac{1}{\nu!} r^\nu e^{-r} {\bf 1}_{[0,\infty]}(r)dr
$$
and normalized in such way that
$$
\int_{0}^\infty L_m^\nu(r)^2 d\mu(r) = \binom{\nu+m} m = L_m^\nu(0).
$$
Let $L^2_{\rad}(\r^d)$ be the subspace of radial functions in $L^2(\r^d)$. For now on we let $\nu=d/2-1$. In this way the functions
\begin{equation}\label{def:lagbasis}
G_m(x) = L_m^{\nu}(2\pi |x|^2)e^{-\pi |x|^2},
\end{equation}
form an orthogonal basis of $L^2_\rad(\r^d)$ and are eigenfunctions of the Fourier transform
$$
\ft G_m(\xi)=(-1)^m G_m(\xi).
$$
\begin{lemma}\label{lemma:flow-lague-herm}
For all $\bn\in\z^d_+$ and $m\in \z_+$ we have
\begin{align*}
& e^{it\Delta } (F_{\bn})(x) \\ & = (1+4\pi it)^{-\frac{d}2}\left({\frac{1-4\pi it}{1+4\pi it}}\right)^{\frac{|\bn|_1}2} F_\bn\left(\frac{x}{\sqrt{1+16\pi^2 t^2}}\right)\exp\left[\frac{4\pi^2it |x|^2}{1+16\pi^2 t^2}\right].
\end{align*}
and
\begin{align*}
& e^{it\Delta } (G_{m})(x) \\ & = (1+4\pi it)^{-\frac d 2}\left(\frac{1-4\pi it }{1+4\pi it}\right)^{m} G_m\left(\frac{x}{\sqrt{1+16\pi^2 t^2}}\right)\exp\left[\frac{4\pi^2 it |x|^2}{1+16\pi^2 t^2}\right].
\end{align*}
\end{lemma}
\begin{proof}
This is \cite[Lemma 11]{Gon}.
\end{proof}

We now are ready to define the Lens transform. For a given function $u:\r^d\times \r\to\cp$ and a given $p>0$ we define the Lens transform of $u$ by
$$
\L u(y,s)=u\left(\frac{y}{\cos(\pi s)\sqrt{2\pi p}},\frac{\tan(\pi s)}{4\pi} \right)(1+i\tan(\pi s))^{\frac{d}2}e^{\frac{1-i\tan(\pi s)}{2p}|y|^2}.
$$
Define the measures
$$
d\be(s)={\boldsymbol 1}_{[-1/2,1/2]}(s)ds
$$ 
and the normal Gaussian distribution in $\r^d$
$$
d\ga(y)=(2\pi)^{-\frac{d}2}e^{-|y|^2/2}dy.
$$
Using the change of variables $y=(2\pi p)^{\frac12}{\cos(\pi s)}x$ and $s=\pi ^{-1}\arctan(4\pi t)$, a routine computation shows that
\begin{align}\label{id:changeofvar}
\|u(x,t)\|_{L^p(\r^{d}\times \r)}= p^{-\frac{d}{2p}}2^{-\frac{2}{p}} \|\L u(y,s)\|_{L^p(d\ga(y)d\be(s))}.
\end{align}
We also define a radial version of the Lens transform for any function $u(|x|,t)$, radial in the variable $x$, by
$$
\L_\rad u(r,s) = u\left(\frac{\sqrt{r}}{\cos(\pi s)\sqrt{2\pi p}},\frac{\tan(\pi s)}{4\pi} \right)(1+i\tan(\pi s))^{\frac{d}2}e^{\frac{1-i\tan(\pi s)}{2p}r}.
$$
A similar change of variables coupled with integration in radial coordinates shows
$$
\|u(x,t)\|_{L^p(\r^{d}\times \r)} =  p^{-\frac{d}{2p}}2^{-\frac{2}{p}} \|\L_\rad u(r,s)\|_{L^p(\r^d,\mu(r))d\be(s))}.
$$
The reason why these definitions are useful and well adapted to our paper (and differ slightly from previous definitions; see \cite{Tao2}) is the following identities, which are implicit in the proofs of \cite[Theorems ~1 and ~5]{Gon}
\begin{align}\label{id:LensofFn}
\L(e^{it\Delta}(F_\bn)(x))(y,s)= H_{\bn}(\sqrt{\tfrac2 p} \, y)e^{-\pi i |\bn|_1 s},
\end{align}
where $H_\bn(x)=H_{n_1}(x_1)...H_{n_d}(x_d)$, and
\begin{align}\label{id:LensofGm}
\L_\rad(e^{it\Delta}(G_m))(r,s)  = L_m^\nu(\tfrac2 p \, r)e^{-2\pi i m s}.
\end{align}
In particular, 
$$\L(e^{-\pi|\cdot|^2})=\L_{\rad}(e^{-\pi|\cdot|^2})=1.$$

We note that the first author derived the Lens transform in \cite{Gon} (without previous knowledge of it) as a way of transforming the Schr\"odinger propagator, in view of Lemma \ref{lemma:flow-lague-herm}, via a change of variables, into a flow over polynomials. Moreover, the fact that the orthonormal basis $\{e^{-2\pi i n s}\}$ appears in the above identity (not quite in \eqref{id:LensofFn}) and that $d\be(s)$ is the indicator functions of the interval $[-1/2,1/2]$ is what makes this version of the Lens transformation a useful tool, it introduces a very convenient orthogonality. 

The following lemmas are crucial.

\begin{lemma}\label{lemma:deficit-hermite}
For
$$
f(x)=\sum_{\bn \in \z_+^d} a(\bn) F_\bn(x) \in L^2(\r^d)
$$
let 
$$
\J f(y)= \sum_{\bn \in \z_+^d} (-i)^{|n|_1} a(\bn) H_\bn(y) \in L^2(d\ga)
$$
Then $2^{-\frac{d}{4}}\J:L^2(\r^d)\to L^2(d\ga)$ is an isometry and
$$
\J (T_{e^{-\pi |\cdot|^2}} \GG)={\rm span}_{\cp}\{1,x_1,x_2,...,x_d,|x|^2\}.
$$
Moreover, if $u(x,t)$ solves the Schr\"odinger equation with initial data $f$ then
\begin{align*}
\phi(f) = 2^{-\frac{2d}{2 + d}} (2 + 4/d)^{-\frac{d^2}{4 + 2 d}}\left(\|\J f\|^2_{L^2(d\ga)} - \|\L(u)(y,2s)\|_{L^p(d\ga(y)d\be(s))}^2\right)
\end{align*} 
\end{lemma}
\begin{proof}
This is \cite[Theorem ~1]{Gon} together with \eqref{id:LensofFn} and \eqref{eq:tangGG}.
\end{proof}

\begin{lemma}\label{lemma:deficit-laguerre}
For
$$
f(x)=\sum_{m \in \z_+} a(m) G_m(x) \in L^2(\r^d)
$$
let 
$$
\J_{\rm rad} f(r)=\sum_{m \in \z_+} a(m) L_m^\nu(r) \in L^2(d\mu).
$$
Then $2^{-\frac{d}{4}}\J_{\rm rad}:L^2_{\rm rad}(\r^d)\to L^2(d\mu)$ is an isometry and
$$
\J_{\rm rad} (T_{e^{-\pi |\cdot|^2}} (\GG\cap L^2_{\rm rad}(\r^2)))={\rm span}_{\cp}\{1,|x|^2\}.
$$
Moreover, if $u(x,t)$ solves the Schr\"odinger equation with initial data $f$ then
\begin{align*}
\phi(f) = 2^{-\frac{2d}{2 + d}} (2 + 4/d)^{-\frac{d^2}{4 + 2 d}}\left(\|\J_{\rm rad}f\|^2_{L^2(d\mu)} - \|\L_{\rm rad}(u)(r,s)\|_{L^p(d\mu(r)d\be(s))}^2\right)
\end{align*}
\end{lemma}
\begin{proof}
This is \cite[Theorem ~5]{Gon} in conjunction with \eqref{id:LensofGm} and \eqref{eq:tangGG}.
\end{proof}

\subsection{Proof of Theorem \ref{thm:main_schro}}

\begin{proof}
{\bf Step 1.} First we show that Gaussians satisfy the Euler-Lagrange equations, that is, Gaussians are critical points of the deficit function $\phi$. For a given $g\in L^2(d\ga)$ let $f=\J^{-1}(g)$ and $u$ solve Sch\"rodinger's equation with initial data $f$. Define
$$
\HH^s g(y) = \L(u)(y,2s)
$$
and 
$$
 \phi_{lens}(g)=  \|g\|^2_{L^2(d\ga)} - \|\HH^s g\|_{L^p(d\mu(r)d\be(s))}^2.
$$
Due to \eqref{id:LensofFn} and Lemma \ref{lemma:deficit-hermite} we have
$$
\HH^s(H_\bn)(y)=H_\bn(\sqrt{\tfrac{2}{p}}\, y)e^{-2\pi i |\bn|_1 s}
$$
and
$$
\phi_{lens}(g)=2^{\frac{2d}{2 + d}} (2 + 4/d)^{\frac{d^2}{4 + 2 d}} \phi(f).
$$
We can now use Appendix \ref{app:abstvar} to obtain
\begin{align*}
 2^{\frac{2d}{2 + d}} (2 + 4/d)^{\frac{d^2}{4 + 2 d}} \phi'(e^{-\pi|\cdot|^2})f  & = \phi_{lens}'(1)(g,g)
\\ & = 2 \Re \int_{\r^d}g(y)d\ga(y)-2 \, \Re \int_{-\frac12}^{\frac12} \int_{\r^d} \HH^s(g)(y)d\ga(y)ds \\
& = 2\Re \, a(\bo 0) - 2\Re \, a(\bo 0)  = 0.
\end{align*}

\noindent {\bf Step 2: Reduction to radial.} By Theorem \ref{thm:general} it suffices to show
$$
\phi''(e^{-\pi|\cdot|^2})(f,f) \geq c \|f\|^2_{L^2(\r^d)}
$$
for all $f\in L^2(\r^d)$ orthogonal to the tangent space of $\GG$ at $e^{-\pi |\cdot|^2}$. We will now show that the minima of 
$$f\mapsto \|f\|^{-2}_{L^2(\r^d)}\phi''(e^{-\pi|\cdot|^2})(f,f)$$ over all $ (T_{e^{-\pi |\cdot|^2}}\GG)^\perp$ is the same as the minima over $ (T_{e^{-\pi |\cdot|^2}}\GG)^\perp \cap L_{\rm rad}^2(\r^d)$.  Let $f\in (T_{e^{-\pi |\cdot|^2}}\GG)^\perp$. We can apply Lemma \ref{lemma:flow-lague-herm} in conjunction with Appendix \ref{app:abstvar} to obtain
\begin{align*}
\phi''(e^{-\pi|\cdot|^2})(f,f)  =2 \, \A_d^2  \,  \|f\|^2_{L^2(\r^d)} - C \int_\r \int_{\r^d} w(|x|,t)^{p-2}|e^{it\Delta} f(x)|^2dxdt,
\end{align*}
for some $C>0$, where $w(r,t)=(1+16\pi^2 t^2)^{-d/4} e^{-\frac{\pi^2 r^2}{1+16\pi^2t^2}}$. Note that a priori there are two missing terms in the above calculation, however is not hard to show they vanish. It is now enough to show the quadratic form
$$
Q(f)= \int_\r \int_{\r^d} w(|x|,t)^{p-2}|e^{it\Delta} f(x)|^2dxdt
$$
is maximized when $f$ is radial and for that we use the decomposition $\r^d=\s^{d-1}\times \r_+$. 

We let 
$$
H_k=\{g(|x|)Y_k(x) \in L^2(\r^d): Y_k \text{ is a spherical harmonic of degree } k\},
$$
where a spherical harmonic of degree $k$ is an homogenous polynomial of degree $k$ in the kernel of $\Delta$. Let $N_{d-1}(k)$ be the dimension of the spherical harmonics of degree $k$ and $\{Y_{k,m}\}_{m=1}^{N_{d-1}(k)}$ an orthonormal basis. Then it is classical that the functions
$$
\Psi_{k,m,n}(x):=L^{\nu+k}_n(2\pi |x|^2) e^{-\pi |x|^2}Y_{k,m}(x), \ \ \ (m=1,...,N_{d-1}(k); \ n\geq 0)
$$
form an orthogonal basis of $H_k$, where $\nu=d/2-1$ and $L^{\nu+k}_n(z)$ is the Laguerre polynomial. Note that $H_0=L_{\rm rad}^2(\r^d)$ and $\Psi_{0,0,n}=G_n$, where $G_n$ is the Laguerre basis defined in \eqref{def:lagbasis}. Moreover, an application of the Hecke-Bochner formula \eqref{hamonicextension} and identity \cite[7.421-4]{GR} shows that
$$
\ft {\Psi_{k,m,n}} = (-i)^k (-1)^n \Psi_{k,m,n}.
$$
Using \eqref{eq:tangGG} in conjunction with the orthogonality properties of Laguerre polynomials one can deduce that
$$
\wt H_k :=(T_{e^{-\pi |\cdot|^2}}\GG)^\perp \cap H_k = \overline{{\rm span}}\left\{\Psi_{k,m,n}: k+n\geq 2\right\}.
$$

Let $Tf(x,t)=w(x,t)^{p/2-1}e^{it\Delta}f(x)$, so $Q(f)=\Braket{T^*Tf | f}$. Applying again the Hecke-Bochner formula \eqref{hamonicextension} one can show that $T^*T (\wt H_k) \subset H_k$, which implies that
$$
\sup_{f\in (T_{e^{-\pi |\cdot|^2}}\GG)^\perp\setminus \{0\}}\|f\|_{L^2(\r^d)}^{-2}Q(f) = \sup_{k\geq 0}\sup_{f\in \wt H_k\setminus\{0\}}\|f\|_{L^2(\r^d)}^{-2}Q(f).
$$
Now  for $k\geq 1$ fix $f_k(x)=g_k(|x|) Y_k(x) \in \wt H_k\setminus \{0\}$, where $Y_k$ is some spherical harmonic of degree $k$ and unit $L^2(\s^{d-1})$-norm. Let 
$$
f^\ep_0(x)=g_k(|x|)|x|^{k}|\s^{d-1}|^{-1/2} + \ep \Psi_{0,0,2}(x)\in H_0, \ \ \   (\text{for small } \ep >0)
$$
and note $\|f_k\|_{L^2(\r^d)}=\|f_0^\ep\|_{L^2(\r^d)}+O(\ep)$. Apriori, $f_0^\ep$ may not be in $\wt H_0$, but in Step ~4 below we show that $Q$ further decomposes over $H_0$ in the form
$$
Q\left(\sum_{n\geq 0} a(n)\Psi_{0,0,n}\right) = \sum_{n\geq 2} |a(n)|^2Q(\Psi_{0,0,n}).
$$
Therefore, if $P:H_0\to\wt H_0$ is the orthogonal projection, $P(\Psi_{0,0,n})=\Psi_{0,0,n}$ if $n\geq 2$ and $P(\Psi_{0,0,n})=0$ otherwise, we see that $Q\circ P = Q$ in $\wt H_0$.  We obtain that $P(f_0^\ep)\neq 0$, for small $\ep>0$, and
$$
\frac{Q(f^\ep_0)}{\|f^\ep_0\|_{L^2(\r^d)}^2}=\frac{Q(P(f^\ep_0))}{\|f^\ep_0\|_{L^2(\r^d)}^2} \leq \frac{Q(P(f^\ep_0))}{\|P (f^\ep_0)\|_{L^2(\r^d)}^2} \leq \sup_{f\in \wt H_0\setminus\{0\}}\|f\|_{L^2(\r^d)}^{-2}Q(f).
$$
Taking $\ep\to 0$ we conclude that for ${\|f_0\|_{L^2(\r^d)}^{-2}}{Q(f_0)} \leq \sup_{f\in \wt H_0\setminus\{0\}}\|f\|_{L^2(\r^d)}^{-2}Q(f)$, where $f_0=f_0^0$. Since $\|f_k\|_{L^2(\r^d)}=\|f_0\|_{L^2(\r^d)}$ is now enough to show that $Q(f_k) \leq Q(f_0)$.

\noindent {\bf Step 3: Reduction to radial continued.} By the Hecke-Bochner formula \eqref{hamonicextension} we have $\ft f_k(x)=h_k(|x|)Y_k(x)$. We first observe that
\begin{align*}
& Q(f_k) \\ &  =\int_\r\int_{\r^d} w(|x|,t)^{p-2}|e^{it\Delta}f_k(x)|^2dxdt 
\\ & = \int_\r\int_{\r^d} (\ft{ w(|\cdot|,t)^{p-2} }*e^{-4\pi^2it |\cdot|^2}\ft{f_k} )(x) e^{4\pi^2 it |x|^2 }\ov{\ft f_k(x)}dxdt 
\\& = \int_\r  \frac{(1+16\pi^2 t^2)^{\frac{d}2-(p-2)\frac{d}4}}{(\pi(p-2))^{\frac{d}2}} dt \int_{\r^{2d}} e^{-\frac{1+16\pi^2t^2}{p-2}|x_1-x_2|^2} e^{-4\pi^2it |x_2|^2}\ft{f_k}(x_2) e^{4\pi^2 it |x_1|^2 }\ov{\ft f_k(x_1)}   dx_1dx_2
\\ & =\int_\r  \frac{(1+ t^2)^{\frac{d}2-(p-2)\frac{d}4}}{{4\pi(\pi(p-2))^{\frac{d}2}}} dt \int_{\r^{2d}} e^{-\frac{1+t^2}{p-2}|x_1-x_2|^2} e^{-\pi it |x_2|^2}\ft{f_k}(x_2) e^{\pi it |x_1|^2 }\ov{\ft f_k(x_1)}   dx_1dx_2
\\ & =\int_\r  \frac{(1+ t^2)^{\frac{d}2-(p-2)\frac{d}4}}{{4\pi(\pi(p-2))^{\frac{d}2}}} Q_t(\ft f_k) dt,
\end{align*}
where 
$$
Q_t(\ft f_k) = \int_{\r^d}\int_{\r^d} e^{-\frac{1+t^2}{p-2}|x_1-x_2|^2} e^{-\pi it |x_2|^2}\ft{f_k}(x_2) e^{\pi it |x_1|^2 }\ov{\ft f_k(x_1)}   dx_1dx_2.
$$
We will now proceed by showing that in fact $Q_t(\ft f_k) \leq Q_t(\ft f_0)$ for all $t\in\r$, which straightforwardly implies $Q(f_k)\leq Q(f_0)$. At this point we recall the Funk-Hecke Formula
\begin{align}\label{funkhecke}
\int_{\s^{d-1}} u(\eta \cdot \xi)Y_k(\xi)d\si(\xi) = \frac{Y_k(\eta)|\s^{d-2}|}{C_k^\nu(1)}\int_{-1}^1 {C_k^\nu(\al)}u(\al) (1-\al^2)^{\nu-1/2}d\al,
\end{align}
where $C_k^\nu(\al)$ is the Gegenbauer polynomial, $\nu=d/2-1$ and $|\eta|=1$. We  obtain
\begin{align*}
& \int_{\r^d} e^{-(\frac{1+t^2}{p-2}+\pi it)|x_1|^2} e^{\frac{1+t^2}{p-2}x_1\cdot x_2}\ft{f_k}(x_2)dx_2  
\\ & = \frac{|\s^{d-2}|Y_k(x_1/|x_1|)}{C_k^\nu(1)}\int_{0}^\infty \int_{-1}^1 e^{-(\frac{1+t^2}{p-2}+\pi it)r_2^2}r_2^{k+d-1} h_k(r_2) e^{\frac{1+t^2}{p-2}r_2|x_1|\al}{C_k^\nu(\al)}(1-\al^2)^{\nu-1/2}d\al dr_2,
\end{align*}
which implies
\begin{align*}
& Q_t(f_k)  = {|\s^{d-2}|}\int_0^\infty \int_0^\infty \ov{h_k(r_1)}{h_k(r_2)}(r_1r_2)^{k+d-1}a_k(t,r_1,r_2)dr_1dr_2,
\end{align*}
where
\begin{align*}
a_k(t,r_1,r_2)  =\int_{-1}^1 e^{-(\frac{1+t^2}{p-2}+\pi it)r_1^2-(\frac{1+t^2}{p-2}-\pi it)r_2^2+\frac{1+t^2}{p-2}r_2r_1\al}\frac{C_k^\nu(\al)}{C_k^\nu(1)}(1-\al^2)^{\nu-1/2}d\al.
\end{align*}
A similar computation leads to
$$
Q_t(f_0)={|\s^{d-2}|}\int_0^\infty \int_0^\infty h_k(r)\ov{h_k(s)}(rs)^{k+d-1}a_0(t,r,s)drds,
$$
where
$$
a_0(t,r_1,r_2) = \int_{-1}^1 e^{-(\frac{1+t^2}{p-2}+\pi it)r_1^2-(\frac{1+t^2}{p-2}-\pi it)r_2^2+\frac{1+t^2}{p-2}r_2r_1\al} (1-\al^2)^{\nu-1/2}d\al.
$$
It is now enough to show that for every $t\in\r$ the kernel
\begin{align*}
b_k(t,r_1,r_2)  & :=a_0(t,r,s)-a_k(t,r,s) \\ & =  \int_{-1}^1 e^{-(\frac{1+t^2}{p-2}+\pi it)r_1^2-(\frac{1+t^2}{p-2}-\pi it)r_2^2+\frac{1+t^2}{p-2}r_2r_1\al} \left[1-\frac{C_k^\nu(\al)}{C_k^\nu(1)}\right](1-\al^2)^{\nu-1/2}d\al
\end{align*}
is positive semi-definite for $(r_1,r_2)\in \r_+^2$.  It is well-known that $\left|\frac{C_k^\nu(\al)}{C_k^\nu(1)}\right|\leq 1$ for $-1<\al<1$, and thus $b_k(t,r_1,r_2)$ is a continuous positive linear combination (in the parameter $\al$) of the kernels 
\begin{equation*}
e^{-(\frac{1+t^2}{p-2}+\pi it)r_1^2-(\frac{1+t^2}{p-2}-\pi it)r_2^2}\times e^{\frac{1+t^2}{p-2}r_2r_1\al}
\end{equation*}
and these are visibly positive definite for $(r_1,r_2)\in \r_+^2$ (being a product of positive definite kernels). Thus $b_k(t,r_1,r_2)$ is positive definite as desired.

\noindent {\bf Step 4.} By the Step ~2 and ~3 it is sufficient to show that
$$
\phi''(e^{-\pi|\cdot|^2})(f,f) \geq c \|f\|^2_{L^2(\r^d)}
$$
for all $f\in L_{\rm rad}^2(\r^d)$ orthogonal to the tangent space of $\GG$ at $e^{-\pi |\cdot|^2}$ and for that we use the radial Lens transform.
For a given  $f\in L^2_{\rm rad}(\r^d)$ with expansion
$$
f(x)=\sum_{m\geq 0} a(m)L_m^\nu(2\pi |x|^2)e^{-\pi |x|^2},
$$
where $\nu=d/2-1$, let 
$$
g(r)=\J_{\rm rad}(f)(r)=\sum_{m\geq 0} a(m)L_m^\nu(r) \in L^2(d\mu).
$$
Let $u$ solve Sch\"rodinger's equation with initial data $f$. Define
$$
\RR^s g(r) = \L_{\rm rad}(u)(r,s)
$$
and 
$$
 \phi_{\rm rad.lens}(g)=  \|g\|^2_{L^2(d\mu)} - \|\RR^s g\|_{L^p(d\mu(r)d\be(s))}^2.
$$
Due to \eqref{id:LensofGm} and Lemma \ref{lemma:deficit-laguerre} we have
$$
\RR^s(L^\nu_m)(r)=L^\nu_m({\tfrac{2}{p}}\, r)e^{-2\pi i m s},
$$
and
$$
\phi_{\rm rad.lens}(g)=2^{\frac{2d}{2 + d}} (2 + 4/d)^{\frac{d^2}{4 + 2 d}} \phi(f).
$$
Using Lemma \ref{lemma:deficit-laguerre} in conjunction with Appendix \ref{app:abstvar} we deduce
\begin{align*}
2^{\frac{2d}{2 + d}} (2 + 4/d)^{\frac{d^2}{4 + 2 d}}\phi''(e^{-\pi|\cdot|^2})(f,f)  &  =  \phi''_{\rm rad.lens}(1)(g,g) \\ 
 & = 2\|g\|_{L^2(d\mu)}^2 - p\int_{-1/2}^{1/2} \int_0^\infty |\RR^s(g)|^2d\mu \\
 & = 2\sum_{m\geq 2}(1-c_m) |a(m)|^2 L_n^\nu(0),
\end{align*}
where
$$
c_m=L_m^\nu(0)^{-1}\frac{p}{2}\int_0^\infty L_m^\nu(\tfrac{2}{p} r)^2 d\mu(r).
$$
Note that by Appendix \ref{app:abstvar} there are some other terms in the second variation that did not show up above, but they can easily be shown to be zero. All we now need to prove is that
$$
{c_m} \leq 1-\ep
$$
for all $m\geq 2$ and some $\ep>0$, which is the content of the next step.

\noindent {\bf Step 5.} First we use the following identity
\begin{equation}\label{id:lagidlambda}
\frac{L_n^\nu(\la r)}{L_n^\nu(0)} = \sum_{j=0}^n \binom n j (1-\la)^{n-j}\la^j \frac{L_j^\nu(r)}{L_j^\nu(0)}
\end{equation}
to obtain
\begin{align}\label{cmsummformula}
{c_m} & = \frac{p}{2} \sum_{j=0}^m \binom{m+\nu}{m-j} \binom m j (1-2/p)^{2m-2j}(2/p)^{2j}.
\end{align}
Identity \eqref{id:lagidlambda} can be easily derived using the generating function
$$
\sum_{n\geq 0} w^n L_n^\nu(r) = \frac{e^{-rw/(1-w)}}{(1-w)^{\nu+1}}.
$$
To obtain an asymptotic expansion for $c_m$ we separate the cases $p\neq 4$ and $p=4$. If $p=4$ (hence $d=2$) we obtain
\begin{align*}
{c_m} = 4^{-m}\frac{p}{2} \binom {2m+\nu} m \sim \frac{p2^{\nu-1}}{\sqrt{\pi m}},
\end{align*}
hence $c_m = O(m^{-1/2})$. For $p\neq 4$ we observe that the series in \eqref{cmsummformula} coincides with the series of a Jacobi polynomial, that is, 
$$
{c_m} = \frac{p}{2} Z^{-m}P_m^{(\nu,0)}(X),
$$
where $P^{(\nu,0)}_m(X)$ is the a Jacobi polynomial, $Z=(1-4/p)^{-1}$ and $X=\tfrac12 Z + \tfrac12 Z^{-1}$. Noticing that $|Z|>1$, we can then apply \cite[Theorem 8.21.9]{Sz} to obtain
\begin{align*}
Z^{-m}P_m^{(\nu,0)}(X) \sim \frac{\kappa}{\sqrt{m}},
\end{align*}
for some explicit $\kappa>0$ depending only on $Z$ and $\nu$. We again obtain $c_m= O(m^{-1/2})$.  

To finish the proof is now enough to show that $c_m < 1$ for all $m\geq 2$. First we introduce the variable $Y=\nu+1$. Using that
$$
\binom {m+\nu}{m-j} = \frac{1}{(m-j)!}\sum_{\ell=0}^{m-j} \sigma_{m-j,\ell}(j,...,m-1)Y^{m-j-\ell},
$$
where $\sigma_{m-j,\ell}$ is the $\ell$-th symmetric function in $m-j$ variables we obtain
\begin{align*}
Y(1+Y)^{2m-1}c_m & = \sum_{j=0}^m \sum_{\ell=0}^{m-j} \binom m j \frac{1}{(m-j)!} \sigma_{m-j,\ell}(j,...,m-1)Y^{m+j-\ell}  \\
& =  \sum_{k=1}^{2m} Y^k \sum_{j=\max\{0,k-m\}}^{\min\{m,\lfloor k/2\rfloor\}} \binom m j \frac{1}{(m-j)!} \sigma_{m-j,m+j-k}(j,...,m-1) \\
& <  \sum_{k=1}^{2m} Y^k \sum_{j=\max\{0,k-m\}}^{\min\{m,\lfloor k/2\rfloor\}}\binom m j \frac{1}{(m-j)!} \binom {m-j} {m+j-k} \frac{(m-1)!}{(k-j-1)!} \\
& =  \sum_{k=1}^{2m} Y^k \sum_{j=\max\{0,k-m\}}^{\min\{m,\lfloor k/2\rfloor\}}\binom m j \binom {m-1} {k-j-1} \frac{1}{(k-2j)!} \\
& < \sum_{k=1}^{2m} Y^k \sum_{j=0}^{\min\{m,k-1\}}\binom m j \binom {m-1} {k-j-1}
\\ & =  \sum_{k=1}^{2m} Y^k \binom {2m-1}{k-1}
\\ & = Y(1+Y)^{2m-1}.
\end{align*}
Above we used that $\sigma_{a,b}(x_1,...,x_a)\leq x_a...x_{a-b+1}\binom{a}{b}$ if $x_a\geq x_{a-1}\geq ... \geq x_{1}$ and that $m\geq 2$, so we get only strict inequalities. This shows $c_m<1$ and finishes the proof.
\end{proof}

\section{Sphere adjoint Fourier restriction}
Here we prove Theorem~\ref{thm:sphere}, concerning the functional 
\begin{equation}\label{eq:SphereDeficitRecall}
    \begin{array}{cc}
\zeta(f) =  \C_d^2\|f\|_{L^2(\s^{d-1})}^2 - \|\ft {f\sigma}\|_{L^p(\r^d)}^2
, & p:=2+\frac4{d-1},
    \end{array}
\end{equation}
where $\C_d :=(2\pi)^{d/2}|\s^{d-1}|^{-\frac{1}{d+1}}(\int_0^\infty |J_{d/2-1}|^{p} \frac{dr}{r^{\frac{d-3}{d-1}}})^{1/p}$ and $\sigma$ is the surface measure on $\mathbb S^{d-1}$. As before, we apply the method of Section~\ref{sec:method}. The relevant manifold is now
$$
\CC:=\{a e^{ix\cdot v} : a\in \cp\setminus\{0\} \text{ and } v\in \r^d\}
\subset L^2(\mathbb S^{d-1}),
$$
and it can be written in the form $\GG=\left\{ z\Gamma(f_\star)\ :\ z\in\mathbb C\setminus \{0\}\, , \Gamma\in \mathcal{F}\right\},$ 
by letting $\mathcal F$ denote the group of symmetries generated by the frequency translations $f(x)\mapsto e^{ib\cdot x}f(x)$, $b\in \r^d$, and letting $f_\star$ denote the constant function $1$. The tangent space is
\begin{equation}\label{eq:tangCC}
T_{1} \CC= {\rm span}_{\mathbb R}\{1, i, ix_j\ :\ j=1,\ldots, d\}.
\end{equation}
The vanishing at infinity property~\eqref{eq:vanish_infinity} of $\F$ is easy to prove. Finally, for $d= 2$, Theorem~\ref{thm:sphere} holds in its stronger, global form, as a consequence of the general Theorem~\ref{thm:general_sharpened}; The necessary pre-compactness of maximizing sequences, as in Definition~\ref{def:ProfDec}, has been proven by Christ and Shao~\cite{ChSh12}, while the uniqueness up to symmetries of the maximizer is due to Foschi~\cite{Fo2}.  
\subsection{Proof of Theorem \ref{thm:sphere}}
{\bf Step 1.} An application of the Funk-Hecke formula \eqref{funkhecke} in conjunction an integral representation of Bessel functions via Fourier transform one can show that for every spherical harmonic $Y_k$ we have
\begin{align}\label{hamonicextension}
\ft{\si Y_k}(r\xi) = \int_{\s^{d-1}} Y_k(\eta) e^{-ix\eta}d\si(\eta)=(2\pi)^{d/2}(-i)^kA_{\nu+k}(r) r^k Y_k(\xi),
\end{align}
where $|\xi|=1$, $r>0$, $\nu=d/2-1$, $A_\nu(z)=J_\nu(z)/z^\nu$ and $J_\nu$ is the Bessel function of the first kind. By Appendix \ref{app:abstvar} we conclude that
\begin{align*}
\zeta'(1)(f) & = 2\S_d^2 \Re \Braket{1|f}_{L^2(\s^{d-1})} - 2\|\ft \sigma\|_{L^p(\r^d)}^{2-p}\Re \int_{\r^d}|\ft \si(x)|^{p-2}\ft \si(x) \overline{\ft {\si f}(x)}dx  = 0-0,
\end{align*}
if $f\in (T_1\CC)^\perp$ (note $1\in T_1\CC$), where $p=2\frac{d+1}{d-1}$. We also have that $f=v\cdot x$ belongs to $T_1\CC$ for any $v\in \r^d$, as the function $f(x)=e^{i\ep x\cdot v}$ belongs to $\CC$ for all $\ep\in \r$ and $\zeta(f)=0$. This implies a remarkable identity
$$
\int_0^\infty |A_\nu(r)|^{p}r^{2\nu+1}dr=(p-1)\int_0^\infty |A_\nu(r)|^{p-2}A_{\nu+1}^2(r)r^{2\nu+3}dr.
$$
Recalling that $Y_1$ is a multiple of $x_1+...+x_d$, the proof of the above identity spoils the excitement since it just amounts to compute $\frac{d^2}{d\ep^2}\zeta(e^{i\ep(x_1+...x_d)})|_{\ep=0}=\zeta''(1)(Y_1,Y_1)=0$, which vanishes since $\zeta(e^{i\ep(x_1+...x_d)})=0$ for all $\ep\in\cp$. In fact, by the same argument the above identity must be true as long as both sides are finite, that is, $p>\frac{2d}{d-1}$.

Theorem \ref{thm:general} now amounts to show that the quadratic form $\zeta''(1)$ is coercive on $(T_1\CC)^\perp$. It is easy to see that this quadratic form is naturally in block form because of formula \eqref{hamonicextension}. Let $f=\sum_{k\geq 2} a_k Y_k \in (T_1 \CC)^\perp$, where $Y_k$ is an orthonormal set (for the surface measure of $\s^{d-1}$) of real spherical harmonics. Define the numbers
$$
c_k = \int_0^\infty |A_\nu(r)|^{p-2}A_{\nu+k}^2(r)r^{2\nu+1+2k}dr.
$$
Noticing that by \eqref{def:Sconst} we have $\|\ft \sigma\|_{L^p(\r^d)}^{2-p}=(2\pi)^{-dp/2}c_0^{-1}\S_d^{2}$, we can use Appendix \ref{app:abstvar} to obtain
\begin{align*}
& c_0\S_d^{-2}\zeta''(1)(f,f) 
\\ & =   2c_0\|f\|^2_{L^2(\s^{d-1})}  - (2\pi)^{-pd/2}\left[p\int_{\r^d} |\ft \si|^{p-2}|\ft{\si f}|^2 + {(p-2)}\Re\int_{\r^d} |\ft \si|^{p-4}(\overline{\ft \si})^2(\ft{\si f})^2\right] \\
& = 2 c_0\sum_{k\geq 2} |a_k|^2 -  \sum_{k\geq 0} c_k(p|a_k|^2  +(p-2) (-1)^k\Re (a_k^2))
\\ & \geq 2\sum_{k\geq 2} (c_0-(p-1)c_k)|a_k|^2
\end{align*}
Therefore we have to show that $(1-\ep)c_0>(p-1)c_k$ for all $k\geq 2$ and some $0<\ep<1$.

\noindent {\bf Step 2.}  The strategy now amounts to derive a decreasing upper bound for $c_k$ and showing that $(p-1)c_k<c_0$ for $k\geq k_0$, where $k_0$ is small but depends on $d$, and then verify each case $k<k_0$ by numerical integration.

\begin{table}[h!]
\centering

\begin{tabular}{ ||c|c|} 
 \hline
 $d$ & $<\widetilde{c}_0$\\
 \hline
2 &  0.33677 \\ 
 \hline
\end{tabular}
\smallskip

\begin{tabular}{ ||c|c|c|} 
 \hline
 $d$ &  ${(p-1)b_8}<$ & $<\widetilde{c}_0$\\
 \hline
3 & 0.29767 & 0.31822 \\ 
 \hline
\end{tabular}
\smallskip

\begin{tabular}{ ||c|c|c|} 
 \hline
 $d$ &  ${(p-1)b_5}<$ & $<\widetilde{c}_0$\\
 \hline
4 & 0.25084 & 0.25803 \\ 
5 & 0.18775 & 0.21158 \\
 \hline
\end{tabular}
\smallskip

\begin{tabular}{ ||c|c|c||c|c|c||c|c|c|} 
 \hline
 $d$ &  ${(p-1)b_4}<$ & $<\widetilde{c}_0$  &  $d$ & ${(p-1)b_4}<$ & $<\widetilde{c}_0 $ &  $d$ & ${(p-1)b_4}<$ & $<\widetilde{c}_0$\\
 \hline
6 & 0.17363 & 0.17776 & 25 & 0.03979 & 0.04134 & 44 & 0.02272 & 0.02311 \\ 
7 & 0.14595 & 0.15265 & 26 & 0.03827 & 0.03970 & 45 & 0.02222 & 0.02258 \\ 
8 & 0.12624 & 0.13347 & 27 & 0.03687 & 0.03818 & 46 & 0.02174 & 0.02208 \\ 
9 & 0.11141 & 0.11842 & 28 & 0.03556 & 0.03678 & 47 & 0.02128 & 0.02160 \\ 
10 & 0.09983 & 0.10632 & 29 & 0.03435 & 0.03548 & 48 & 0.02084 & 0.02114 \\ 
11 & 0.09050 & 0.09641 & 30 & 0.03321 & 0.03426 & 49 & 0.02042 & 0.02069 \\ 
12 & 0.08282 & 0.08815 & 31 & 0.03215 & 0.03312 & 50 & 0.02001 & 0.02027 \\ 
13 & 0.07637 & 0.08117 & 32 & 0.03116 & 0.03206 & 51 & 0.01962 & 0.01986 \\ 
14 & 0.07087 & 0.07520 & 33 & 0.03022 & 0.03106 & 52 & 0.01925 & 0.01947 \\ 
15 & 0.06613 & 0.07003 & 34 & 0.02934 & 0.03012 & 53 & 0.01889 & 0.01909 \\ 
16 & 0.06200 & 0.06551 & 35 & 0.02851 & 0.02924 & 54 & 0.01854 & 0.01873 \\ 
17 & 0.05836 & 0.06154 & 36 & 0.02772 & 0.02840 & 55 & 0.01820 & 0.01838 \\ 
18 & 0.05513 & 0.05801 & 37 & 0.02698 & 0.02761 & 56 & 0.01788 & 0.01804 \\ 
19 & 0.05224 & 0.05486 & 38 & 0.02628 & 0.02687 & 57 & 0.01757 & 0.01772 \\ 
20 & 0.04964 & 0.05203 & 39 & 0.02561 & 0.02616 & 58 & 0.01727 & 0.01740 \\ 
21 & 0.04730 & 0.04948 & 40 & 0.02497 & 0.02549 & 59 & 0.01698 & 0.01710 \\ 
22 & 0.04516 & 0.04716 & 41 & 0.02437 & 0.02485 & 60 & 0.01669 & 0.01681 \\ 
23 & 0.04322 & 0.04505 & 42 & 0.02379 & 0.02424 &  &  &  \\ 
24 & 0.04143 & 0.04311 & 43 & 0.02324 & 0.02366 &  &  &  \\
 \hline
\end{tabular}
\bigskip

\caption{Numerical integration of $\widetilde{c}_0$ with error $< 10^{-5}$ and rounded down to ~5 digits. Evaluation of ${(p-1)b_k}$ rounded up to ~5 digits.}
\label{table:1}
\end{table}

%\begin{center}
\begin{table}[h!]
\centering

\begin{tabular}{ ||c|c|c|} 
 \hline
$d$ &  $(p-1) \times c_2,...,c_6< $ & $<\widetilde{c}_0$\\
 \hline
2 & 0.18707, 0.12746, 0.09661, 0.07813, 0.06546 & 0.33677 \\ 
 \hline
\end{tabular}
\smallskip

\begin{tabular}{ ||c|c|c|} 
 \hline
$d$ &  $(p-1) \times c_2,...,c_7< $ & $<\widetilde{c}_0$\\
 \hline
3 & 0.19229, 0.13788, 0.10741, 0.08827, 0.07476, 0.06512 & 0.31822 \\ 
 \hline
\end{tabular}
\smallskip

\begin{tabular}{ ||c|c|c|} 
 \hline
$d$ &  $(p-1) \times c_2,c_3,c_4< $ & $<\widetilde{c}_0$\\
 \hline
4 & 0.17022, 0.12724, 0.10159 & 0.25803 \\ 
5 & 0.14910, 0.11524, 0.09395 & 0.21158\\
 \hline
\end{tabular}
\smallskip

\begin{adjustbox}{width=1\textwidth}
\begin{tabular}{||c|c|c||c|c|c||c|c|c|}
 \hline
 $d$ &  $(p-1) \times c_2,c_3< $ & $<\widetilde{c}_0$ & $d$ &$(p-1) \times c_2,c_3 < $ & $<\widetilde{c}_0$ & $d$ & $(p-1) \times c_2,c_3< $ & $<\widetilde{c}_0$\\
 \hline
6 & 0.13163, 0.10454 & 0.17776 & 25 & 0.03870, 0.03595 & 0.04134 & 44 & 0.02262, 0.02167 & 0.02311 \\ 
7 & 0.11744, 0.09536 & 0.15265 & 26 & 0.03730, 0.03474 & 0.03970 & 45 & 0.02214, 0.02123 & 0.02258 \\ 
8 & 0.10583, 0.08754 & 0.13347 & 27 & 0.03600, 0.03361 & 0.03818 & 46 & 0.02168, 0.02081 & 0.02208 \\ 
9 & 0.09621, 0.08084 & 0.11842 & 28 & 0.03478, 0.03255 & 0.03678 & 47 & 0.02124, 0.02040 & 0.02160 \\ 
10 & 0.08815, 0.07505 & 0.10632 & 29 & 0.03365, 0.03156 & 0.03548 & 48 & 0.02081, 0.02001 & 0.02114 \\ 
11 & 0.08130, 0.07002 & 0.09641 & 30 & 0.03259, 0.03063 & 0.03426 & 49 & 0.02041, 0.01963 & 0.02069 \\ 
12 & 0.07541, 0.06561 & 0.08815 & 31 & 0.03159, 0.02974 & 0.03312 & 50 & 0.02001, 0.01927 & 0.02027 \\ 
13 & 0.07031, 0.06171 & 0.08117 & 32 & 0.03065, 0.02891 & 0.03206 & 51 & 0.01964, 0.01892 & 0.01986 \\ 
14 & 0.06585, 0.05824 & 0.07520 & 33 & 0.02977, 0.02813 & 0.03106 & 52 & 0.01927, 0.01859 & 0.01947 \\ 
15 & 0.06191, 0.05514 & 0.07003 & 34 & 0.02894, 0.02739 & 0.03012 & 53 & 0.01893, 0.01826 & 0.01909 \\ 
16 & 0.05841, 0.05235 & 0.06551 & 35 & 0.02815, 0.02668 & 0.02924 & 54 & 0.01859, 0.01795 & 0.01873 \\ 
17 & 0.05529, 0.04982 & 0.06154 & 36 & 0.02741, 0.02601 & 0.02840 & 55 & 0.01826, 0.01765 & 0.01838 \\ 
18 & 0.05248, 0.04753 & 0.05801 & 37 & 0.02670, 0.02537 & 0.02761 & 56 & 0.01795, 0.01735 & 0.01804 \\ 
19 & 0.04994, 0.04544 & 0.05486 & 38 & 0.02603, 0.02477 & 0.02687 & 57 & 0.01765, 0.01707 & 0.01772 \\ 
20 & 0.04763, 0.04353 & 0.05203 & 39 & 0.02539, 0.02419 & 0.02616 & 58 & 0.01736, 0.01680 & 0.01740 \\ 
21 & 0.04553, 0.04176 & 0.04948 & 40 & 0.02478, 0.02364 & 0.02549 & 59 & 0.01707, 0.01653 & 0.01710 \\ 
22 & 0.04361, 0.04014 & 0.04716 & 41 & 0.02421, 0.02312 & 0.02485 & 60 & 0.01680, 0.01628 & 0.01681 \\ 
23 & 0.04184, 0.03864 & 0.04505 & 42 & 0.02365, 0.02261 & 0.02424 & & & \\ 
24 & 0.04021, 0.03725 & 0.04311 & 43 & 0.02313, 0.02213 & 0.02366 & & & \\
 \hline
\end{tabular}
\end{adjustbox}
\bigskip

\caption{We use the bound $|J_\nu(r)|\leq r^{-\frac12}$ for $r\geq \frac32 \nu$ and $\nu \geq \frac12$, and  $|J_0(r)|\leq {r^{-\frac12}}$ for $r>0$ (see \cite[Lemma 8]{COS} and \cite[Corollary 2.8(a)]{OT} respectively), which shows that
$
\int_{2000}^\infty  |A_\nu(r)|^{p-2}A_{\nu+k}^2(r)r^{2\nu+1+2k}dr \leq  \int_{2000}^\infty  r^{-2}dr = 5\times 10^{-4}
$
if $d\geq 2$ and $k\geq 2$. We then numerically evaluate the integral  $\widetilde{c}_k:= \int_0^{2000}  |A_\nu(r)|^{p-2}A_{\nu+k}^2(r)r^{2\nu+1+2k}dr$ with error $\leq 10^{-5}$, and thus
$$
(p-1){c}_k\leq (p-1)(\widetilde{c}^{\text{ numeric}}_k+10^{-5}+5\times 10^{-4}).
$$
We then round up the right hand side quantity to $5$ digits.  Column $\widetilde{c}_0$ is copied from Table \ref{table:1} to make the inequalities visible.}
\label{table:2}
\end{table}

%\newpage

The only inequality we need to use is due to Landau \cite{La}:
$$
L:=0.7857468705 > r^{1/3}|J_{\al}(r)|, \text{ for all } r,\al\geq 0.
$$
We obtain
$$
c_k < L^{p-2}\int_0^\infty J_{\nu+k}(r)^2 \frac{dr}{r^\la}=\frac{L^{p-2}\, \Gamma(\la) \Gamma(\nu+k+ (1 - \la)/2)}{2^\la \Gamma((1 + \la)/2)^2 \Gamma(\nu+k + (1 + \la)/2)} =:b_k
$$
where $\la=\frac{3d-5}{3d-3}$ and the last identity is \cite[Eq. 6.574-2]{GR}. It is not hard to show, using the properties of the Gamma function, that $\frac{d}{dk}b_k<0$, and therefore $b_k$ is a decreasing sequence. Also, a routine application of Stirling's asymptotic formula shows that $b_k \approx k^{-\la}$ as $k\to\infty$, and therefore we only need to show that $(p-1)c_k<c_0$ for all $k\geq 2$.
We obtain
$$
\frac{(p-1)c_k}{c_0} < \frac{(p-1)b_{k'}}{\widetilde{c}_0}, \ \  (2\leq k'\leq k)
$$
where 
$$
\widetilde{c}_0=\int_0^{2000} |J_\nu(r)|^p \frac{dr}{r^{\frac{d-3}{d-1}}}.
$$
Numerical evaluation of ${(p-1)b_k}$ and ${\widetilde{c}_0}$ produces Table \ref{table:1}. It shows that $\frac{(p-1)b_8}{\widetilde{c}_0}<1$ if $d=3$, $\frac{(p-1)b_5}{\widetilde{c}_0}<1$ if $d=4,5$ and $\frac{(p-1)b_4}{\widetilde{c}_0}<1$ if $6\leq d\leq 60$. If $d=2$ the above bound gives $\frac{(p-1)b_{283}}{\widetilde{c}_0}<1$, which is terrible. Instead, we can use \cite[Theorem ~10]{CFOT} to deduce directly that $\frac{(p-1)c_k}{\widetilde{c}_0}<1$ for all $k\geq 7$. Finally, numerical approximation of $(p-1)c_k$ produces Table \ref{table:2}, which shows that $(p-1)c_k<\widetilde{c}_0<c_0$ for $k=2,...,6$ if $d=2$, $k=2,...,7$ if $d=3$, $k=2,3,4$ if $d=4,5$ and $k=2,3$ if $6\leq d\leq 60$. This completes the proof.

%\newpage

% ----- APPENDIX -----

\appendix

\section{First and second variation of the abstract deficit functional}\label{app:abstvar}
We consider here the abstract deficit functional, introduced in Section~\ref{sec:method}; 
\begin{equation}\label{eq:AbstrDefAppendix}
    \begin{array}{cc} 
    \displaystyle
    \psi(f)=
   C_\star^2 \langle f|f\rangle - \left(\int_{\mathbb R^N} |Sf|^p\right)^\frac2p, 
     & \text{where } C_\star = \frac{\lVert Sf_\star\rVert_{L^p}}{\lVert f_\star \rVert}.
    \end{array}
\end{equation}
    Here $p>2$ and $S\colon \mathcal H\to L^p(\mathbb R^N)$ is a bounded operator on the Hilbert space $\mathcal H$, whose scalar product we have denoted by $\langle\cdot|\cdot\rangle$, and whose norm is given by $\norm{f}=\sqrt{\langle f|f\rangle}$. Also, $f_\star$ is a fixed element of $\mathcal H$.
We want to compute the following linear and quadratic forms; 
\begin{equation}\label{eq:VariationRecall}
    \begin{array}{ccc}
        \psi'(f_\star)f:=\left.\frac{\partial}{\partial \epsilon}\right|_{\epsilon=0} \psi(f_\star + \epsilon f), & \psi''(f_\star)(f,f):=\left.\frac{\partial^2}{\partial \epsilon^2}\right|_{\epsilon=0} \psi(f_\star + \epsilon f),
        &\forall f\in \mathcal H.
    \end{array}
\end{equation}
Now, since 
\begin{equation}\label{eq:norm_expansion}
    \langle f_\star + \epsilon f|f_\star + \epsilon f\rangle = \|f_\star\|^2+2\epsilon \Re\langle f_\star |f\rangle +{\epsilon^2}\norm{f}^2, 
\end{equation}
and  
\begin{align}\label{eq:integral_expansion}
    \begin{split}
       & \left(\int_{\mathbb R^N} \lvert S(f_\star + \epsilon f)\rvert^p\right)^\frac2p \\ &= \|Sf_\star\|_{L^p}^2+ 2\epsilon \|Sf_\star\|_{L^p}^{2-p}\Re \int_{\mathbb R^N} \lvert Sf_\star \rvert^{p-2}\overline{S f_\star}Sf \\ 
        & \ \ \, +\frac{\epsilon^2}{2}\|Sf_\star\|_{L^p}^{2-p}\left( p\int_{\mathbb R^N} \abs{Sf_\star}^{p-2}\abs{Sf}^2 + (p-2)\Re \int_{\mathbb R^N} \abs{Sf_\star}^{p-4} (\overline{Sf_\star})^2 (Sf)^2 \right) \\
        & \ \ \, + \frac{\epsilon^2}{2} \|Sf_\star\|_{L^p}^{2-2p}\left[ 2(2-p)\left( \Re\int_{\mathbb R^N} \abs{Sf_\star}^{p-2} \overline{Sf_\star}Sf\right)^2 \right] + o(\epsilon^2),
    \end{split}
\end{align}
we conclude that 
\begin{equation}\label{eq:app_first_variation}
    \psi'(f_\star)f=2C_\star^2\Re\langle f_\star|f\rangle -2\|Sf_\star\|_{L^p}^{2-p}\Re \int_{\mathbb R^N}\abs{Sf_\star}^{p-2} \overline{Sf_\star} Sf, 
\end{equation}
and that 
\begin{equation}\label{eq:app_second_variation}
    \begin{split}
       & \psi''(f_\star)(f, f) \\ & =2C_\star^2\norm{f}^2-  p\|Sf_\star\|_{L^p}^{2-p}\int_{\mathbb R^N} \abs{Sf_\star}^{p-2} \abs{Sf}^2 - (p-2)\|Sf_\star\|_{L^p}^{2-p}\Re \int_{\mathbb R^N} \abs{Sf_\star}^{p-4}(Sf_\star)^2(Sf)^2 \\ 
        & \ \ \, -2(2-p)\|Sf_\star\|_{L^p}^{2-2p}\left(\Re  \int_{\mathbb R^N} \abs{ Sf_\star}^{p-2} \overline{S f_\star} Sf\right)^2.
    \end{split}
\end{equation}
We immediately see from these two equations that, if $\psi'(f_\star)f=0$ and $\Re \langle f_\star| f\rangle =0$, then the last term in~\eqref{eq:app_second_variation} vanishes.

\section*{Acknowledgements}
 F.\@G.\@ acknowledges support from the Deutsche Forschungsgemeinschaft through the Collaborative Research Center 1060. G.\@N.\@ acknowledges support from the EPSRC New Investigator Award \emph{Sharp Fourier  Restriction  Theory},  grant  no. ~EP/T001364/1, and from the ERCEA Adv. Grant 2014 669689-HADE. Both authors are grateful to Diogo Oliveira e Silva for the stimulating conversations.

\end{document}